\documentclass[11pt]{amsart}
\usepackage[leqno]{amsmath}
\usepackage[svgnames]{xcolor}
\usepackage{tikz}
\usetikzlibrary{arrows}
\usepackage{amsfonts, amssymb, amsopn} \usepackage[utf8]{inputenc}
\usepackage{hyperref}

\usepackage{tikz}
\usetikzlibrary{shapes.geometric}
\usetikzlibrary{calc}
\usetikzlibrary{positioning}
\usetikzlibrary{intersections}

\tikzstyle{stpoint} =
  [ fill=black,
    inner sep=0pt,
    minimum size=4pt,
    circle
  ]
\tikzstyle{stsmallpoint} =
  [ fill=black,
    inner sep=0pt,
    minimum size=2pt,
    circle
  ]

\newcommand{\point}[1]{\node[circle,inner sep = 0pt,minimum size =0pt]}

\usepackage{graphicx}
\usepackage{xspace}

\usepackage{epsfig}
\usepackage{color}
\usepackage{amsmath}

\newcommand{\constant}{19}

\newtheorem{lemma}{Lemma}
\newtheorem{proposition}{Proposition}
\newtheorem{theorem}{Theorem}
\newtheorem{corollary}{Corollary}

\newtheorem{claim}{Claim}
\newtheorem{claimprop3}{Claim}
\theoremstyle{definition}

\newtheorem{remark}{Remark}

\newtheorem{question}{Question}[section]

   \newcommand{\commentout}[1]{}

\DeclareMathOperator{\Sh}{Sh}
\DeclareMathOperator{\St}{Sh}
\DeclareMathOperator{\Sp}{C}

\begin{document}

\thispagestyle{empty}
\centerline{\Large\bf Packing and covering with balls on Busemann surfaces
  \footnote{The final publication is available at Springer via \url{http://dx.doi.org/10.1007/s00454-017-9872-0}} 
}

\vspace{10mm}

\centerline{{\sc Victor Chepoi}, {\sc Bertrand Estellon}, and {\sc Guyslain Naves}}

\vspace{3mm}

\date{\today}

\medskip
\begin{small}
\centerline{Aix Marseille Université, CNRS, LIF UMR 7279, 13288, Marseille, France}

\centerline{\texttt{\{victor.chepoi, bertrand.estellon, guyslain.naves\}@lif.univ-mrs.fr}}
\end{small}

\bigskip\bigskip\noindent {\footnotesize {\bf Abstract.} In this note
we prove that for any compact subset $S$ of a Busemann surface
$({\mathcal S},d)$ (in particular, for any simple polygon with
geodesic metric) and any positive number $\delta$, the minimum number
of closed balls of radius $\delta$ with centers at $\mathcal S$ and
covering the set $S$ is at most \constant{} times the maximum number of
disjoint closed balls of radius $\delta$ centered at points of $S$:
$\nu(S)\le \rho(S)\le \constant{}\nu(S)$, where $\rho(S)$ and $\nu(S)$ are the
covering and the packing numbers of $S$ by ${\delta}$-balls.  Busemann surfaces
represent a far-reaching generalization not only of simple polygons,
but also of Euclidean and hyperbolic planes and of all planar
polygonal complexes of global non-positive curvature. Roughly
speaking, a Busemann surface is a geodesic metric space homeomorphic
to ${\mathbb R}^2$ in which the distance function is convex.}

\section{Introduction}

The set packing and the set covering problems are classical questions
in computer science~\cite{Va}, combinatorics~\cite{Be}, and
combinatorial optimization~\cite{Co,Sch}. Packing and covering
problems in ${\mathbb R}^d$ with special geometric objects have been
also actively investigated in computational geometry
~\cite{AgMu,BrGo,ChHP,MuRa} and in discrete geometry~\cite{Bo,PaAg}.
Finally, the covering and packing problems of arbitrary metric spaces
with balls (which is the subject of the current paper) have been
formulated in the middle of 20th century in pure mathematics
~\cite{KoTi}. The respective covering and packing numbers capture the
size of the underlying metric space and play a central role in several
areas of pure and applied mathematics: information theory, functional
analysis, probability theory, statistics, and learning theory
~\cite{Du,KuMi,Lo}.

In the {\it set covering problem}, given a collection $\mathcal F$ of
subsets of a (finite or infinite) domain $X$, the task is to find a
subcollection of $\mathcal F$ of minimum size $\rho({\mathcal F})$
whose union is $X.$ The {\it set packing problem} asks to find a
maximum number $\nu({\mathcal F})$ of pairwise disjoint subsets of
$\mathcal F$. Another problem closely related to set covering is the
hitting set problem. A subset $T$ is called a {\it hitting set} of
$\mathcal F$ if $T\cap S\ne\varnothing$ for any $S\in {\mathcal F}.$ The
{\it minimum hitting set problem} asks to find a hitting set of
$\mathcal S$ of smallest cardinality $\tau ({\mathcal F}).$ All these
three problems are $NP$-hard, moreover, they are difficult to
approximate within a constant factor unless $P=NP$.  In case when $X$
is a metric space and $\mathcal F$ is the set of its balls of equal
radii, then the minimum covering and the
minimum hitting set problems are equivalent, i.e., $\rho({\mathcal
F})=\tau({\mathcal F})$. Indeed, the centers of balls in any covering of $X$
define a hitting set of  $\mathcal F$ and vice-versa, given a hitting set $T$ of
$\mathcal F$ one can define a covering of $X$ of the same size by considering the
balls centered at the points of $X$.

The inequality $\tau({\mathcal F})\ge \nu({\mathcal F})$ holds for any
family of sets $\mathcal F$ on any domain $X$: any two sets from a
packing cannot be hit by the same point of $X$.  Of particular
importance are the families of sets $\mathcal F$ for which there
exists a universal constant $c:=c({\mathcal F})$ such that
$\tau({\mathcal F}')\le c\nu({\mathcal F}')$ holds for any subfamily
${\mathcal F}'$ of $\mathcal F$.  In general, proving that for all
subfamilies of a particular family of sets $\mathcal F$ such a
universal constant $c$ exists is a notoriously difficult problem and
it is open for many simple particular cases.  For example, in 1965,
Wegner~\cite{We} asked if for the family $\mathcal R$ of all
axis-parallel rectangles in ${\mathbb R}^2$ it is always true that
$\tau({\mathcal R})\le 2\nu({\mathcal R})-1$ (Gy\'{a}rf\'{a}s and
Lehel~\cite{GuLe} relaxed this question by asking if $\tau({\mathcal
R})\le c\nu({\mathcal R})$ for a universal constant $c$).

We briefly review now some families $\mathcal F$ for which the
inequality $\tau({\mathcal F})\le c\nu({\mathcal F})$ holds (when
$\mathcal F$ is a family of balls in a metric space some known results
will be reviewed in the next section).  The equality $\tau({\mathcal
F})=\nu({\mathcal F})$ holds if $\mathcal F$ is an interval
hypergraph, a hypertree, and more generally, a normal hypergraph
~\cite{Be,Sch}. Covering and packing problems for special families of
subtrees of a tree have been considered in~\cite{BaEdWo, Sch}. Alon
~\cite{Al1,Al} established that if ${\mathcal F}$ is a family of
$\kappa$-intervals (i.e., unions of at most $\kappa$ intervals) of the line (or a
family consisting of unions of at most $\kappa$ subtrees of a tree), then
$\tau({\mathcal F})\le 2\kappa^2\nu({\mathcal F}).$ A
similar result has been obtained in~\cite{ChEs} for unions of $\kappa$
balls in a geodesic $\delta$-hyperbolic space.  Gy\'{a}rf\'{a}s and
Lehel's relaxation of Wegner's conjecture was confirmed in
~\cite{ChFe,CoFePe-LaSo} for families of axis-parallel rectangles
intersecting a common monotone curve. One common feature of all these
results is that the inequality $\tau({\mathcal F})\le c\nu({\mathcal
F})$ is established by constructing in a primal-dual way a hitting set
$T$ and a packing ${\mathcal P}\subseteq {\mathcal F}$ such that
$|T|\le c|{\mathcal P}|$.  Consequently, this provides a factor $c$
approximation algorithm for hitting set and packing problems for
$\mathcal F$.

In this note, we consider the problem of covering and packing by balls
of equal radii of subsets of Busemann surfaces. Using a similar
approach as above, we prove that the minimum number of closed balls of
radius $\delta$ required to cover a compact subset $S$ of a Busemann
surface $({\mathcal S},d)$ is at most \constant{} times the maximum number of
pairwise disjoint closed balls of radius $\delta$ with centers in
$S$. Our initial motivation was to establish that such an inequality
holds for simple polygons with geodesic metric. Busemann surfaces
represent a far-reaching generalization not only of simple polygons,
but also of Euclidean and hyperbolic planes and of all planar
polygonal complexes of global non-positive curvature. Roughly
speaking, a Busemann surface is a geodesic metric space homeomorphic
to ${\mathbb R}^2$ in which the distance function is convex~\cite{Pa}.

\section{Preliminaries and main results}

In this section, we recall all necessary definitions and results
related to the subject of this paper.  We start with a subsection in which
we recall some definitions, characterizations,  and notations on
geodesic metric spaces, Busemann spaces, and Busemann surfaces. We continue with
two subsections, one dedicated to basic notions and notations about covering and
packing problems, and the second one to some known results on covering
and packing metric spaces and graphs with balls.  We conclude the section
with the formulation of the main results.

\subsection{Busemann surfaces}

We start with definitions of geodesics and geodesic metric spaces, in
which we follow \cite[Chapter I.1]{BrHa} and \cite[Chapter 2]{Pa}.
Let $(X,d)$ be a metric space.  A {\it geodesic path} joining
$x \in X$ to $y \in X$ is a map $\gamma$ from the closed interval
$[0,l] \subset {\mathbb R}$ to $X$ such that
$\gamma(0) = x, \gamma(l) = y$ and
$d(\gamma(t),\gamma(t')) = |t - t'|$ for all $t,t' \in [0,l]$ (in
particular, $l = d(x,y)$).  The image of $\gamma$ is called a {\it
  geodesic segment} (or a {\it geodesic}) with endpoints $x$ and $y$.
Let $[a,b] \subset {\mathbb R}$ be an interval. A map
$\gamma: [a,b] \rightarrow X$ is said to be an {\it affine
  reparametrized geodesic} or a {\it constant speed geodesic}, if
there exists a constant $\lambda$ such that
$d(\gamma(t),\gamma(t')) = \lambda \cdot |t-t'|$ for all
$t,t'\in [a,b]$.

The definitions of (geodesic) lines and (geodesic) rays are similar to
that of geodesic segment: a \emph{geodesic line} (\emph{resp.}
\emph{geodesic ray}) $\gamma$ is a map from $I := \mathbb{R}$
(\emph{resp.} $I := [0,\infty)$) to $X$ such that for all
$t, t' \in I$, $d(\gamma(t),\gamma(t')) = |t - t'|$. We will refer to
the image of $\gamma$ as a geodesic line or geodesic ray. A {\it local
  geodesic} is a map $\gamma$ from an interval
$I \subseteq {\mathbb R}$ to $X$ such that for every $t \in I$ there
exists $\epsilon > 0$ such that the restriction of $\gamma$ on
$I \cap [t-\epsilon,t+\epsilon]$ is geodesic.

A metric space $X$ is said to be a {\it geodesic metric space} if
every pair of points in $X$ can be joined by a geodesic. A {\it
  uniquely geodesic space} is a geodesic space in which every pair of
points can be joined by a unique geodesic.

We continue with the definition of Busemann spaces; we follow
\cite[Chapter 8]{Pa}.  A {\it Busemann space} (or a {\it
  non-positively curved space in the sense of Busemann}) is a geodesic
metric space $(X,d)$ in which the distance function between any two
geodesics is convex: for all affinely reparametrized geodesics
$\gamma: [0,1]\rightarrow X$ and $\gamma': [0,1]\rightarrow
X$, we have for all $t\in [0,1]$,
$d(\gamma(t),\gamma'(t))\le (1-t)\cdot d(\gamma(0),\gamma'(0))+t\cdot d(\gamma(1),\gamma'(1))$.
Equivalently, $(X,d)$ is a Busemann space if for any two affinely
reparametrized geodesics $\gamma: [a,b]\rightarrow X$ and
$\gamma': [a',b']\rightarrow X$ the map
$f_{\gamma,\gamma'}(t):[0,1]\rightarrow {\mathbb R}$ defined by
$f_{\gamma,\gamma'}(t)=d(\gamma((1-t)a+tb),\gamma'((1-t)a'+tb'))$ is a
convex function.  We continue by recalling the following fundamental
characterizations of Busemann surfaces (they constitute a part of
\cite[Proposition 8.1.2]{Pa}):

\medskip\noindent
{\bf Proposition A}. \label{thales} \cite[Proposition 8.1.2(ii)\&(v)\&(vi)]{Pa}
{\it For a geodesic metric space $(X,d)$, the following conditions are equivalent:

\begin{itemize}
\item[(i)] $X$ is a Busemann space;

\item[(ii)] Let $\gamma : [0,l] \to X$ and $\gamma' : [0,l'] \to X$ be
  two arbitrary geodesics in $X$. For every $t \in [0,1]$,
  $d(\gamma(t \cdot l),\gamma'(t \cdot l'))
   \le (1-t) \cdot d(\gamma(0),\gamma'(0)) + t \cdot d(\gamma(l),\gamma'(l'));$

 \item[(iii)] Let $\gamma : [0,l] \to X$ and $\gamma' : [0,l'] \to X$
   be two arbitrary geodesics of $X$.  Then
   $d\left(\gamma\left(\frac{l}{2}\right),\gamma'\left(\frac{l'}{2}\right)\right)
    \le \frac{1}{2}(d(\gamma(0),\gamma'(0)) + d(\gamma(l) + \gamma'(l')));$

\item[(iv)] Let $\gamma : [0,l] \to X$ and
  $\gamma' : [0,l'] \to X$ be two arbitrary geodesics of
  $X$ having a common initial point $\gamma(0) = \gamma'(0)$. For all
  $t \in [0,1]$,
  $d(\gamma(t \cdot l),\gamma'(t \cdot l'))
   \le t \cdot d(\gamma(l),\gamma'(l')).$
\end{itemize}
}

Busemann spaces satisfy many fundamental metric, geometric, and
topological properties: they are contractible, have the fixed point
property, are uniquely geodesic, local geodesics are geodesics, open
and closed balls are convex, projections on convex sets are unique,
and geodesics vary continuously with their endpoints.  They can be
characterized in a pretty local-to-global way: every complete geodesic
locally compact, locally convex and simply connected metric space is a
Busemann space. For these and other results on Busemann spaces consult
the book of Papadopoulos~\cite{Pa}.

Basic examples of Busemann spaces are the Euclidean space ${\mathbb
E}^n$, and more generally, normed strictly convex vector spaces, the
hyperbolic $n$-dimensional space ${\mathbb H}^n$, ${\mathbb R}$-trees,
and Riemannian manifolds of global nonpositive sectional curvature.  A
large subclass of Busemann spaces is constituted by non-positively
curved spaces in the sense of Alexandrov, known also under the name of
CAT(0) spaces~\cite{BrHa}.

A {\it planar surface} (without boundary) $\mathcal S$ is a
2-dimensional manifold homeomorphic to the plane ${\mathbb R}^2$.  A
geodesic metric space $({\mathcal S},d)$ is called a {\it Busemann
  surface} if $\mathcal S$ is a 2-dimensional manifold and the metric
space $({\mathcal S},d)$ is a Busemann space. Since Busemann spaces
are contractible (by convexity of the distance function), each
Busemann surface is a planar surface.

Particular instances of Busemann surfaces are
non-positively curved piecewise-Euclidean (PE) (or piecewise
hyperbolic) planar complexes without boundary. In fact, as is shown in
\cite[Subsection 2.4]{ChChNa}, any finite non-positively curved
planar complex can be extended to a Busemann surface. Recall that a {\it
planar PE complex} $X$ is obtained from a (not necessarily finite)
planar graph $G$ by replacing each inner face of $G$ having $n$ sides
by a convex $n$-gon in the Euclidean plane. The planar PE complex $X$ is
called a {\it non-positively curved planar complex} if the sum of angles
around each inner vertex of $G$ is at least $2\pi$.  Equivalently, by
\cite[Theorem 5.4]{BrHa} $X$ is  non-positively curved if and only if $X$
endowed with the intrinsic $l_2$-metric $d$ is  uniquely geodesic, or,
equivalently, is a
Busemann (or a CAT(0)) space.

Our motivating examples of Busemann surfaces are the simple polygons
$P$ in the plane endowed with the intrinsic geodesic metric. After
triangulating $P$, one can view $P$ as a finite non-positively curved
planar complex and, as noticed in \cite{ChChNa}, $P$ can be extended
to a Busemann surface $\mathcal S$ so that $P$ will be a convex subset
of $\mathcal S$.

To embed a finite non-positively curved planar complex $X$ (or a
triangulated simple polygon) into a Busemann surface $\mathcal S$, to
each boundary edge $e$ of $X$ we add a closed halfplane $H_e$ of
${\mathbb R}^2$ so that $e$ is a segment of the boundary of $H_e$. If
two boundary edges $e,e'$ of $X$ share a common endvertex $x,$ then
$H_e$ and $H_e'$ will be glued along the rays of their boundaries
emanating from $x$ which are disjoint from $e$ and $e'$.  It can be
easily seen that the resulting planar surface $\mathcal S$ is CAT(0)
and that $X$ isometrically embeds into $\mathcal S$.

Several elementary properties of geodesic lines and convex sets in
Busemann planar surfaces have been presented in \cite{ChChNa}. In our
proofs we use some of these properties (convexity of cones and
triangles, Pasch and Peano axioms, geodesic extension property), which
will be recalled together with some basic properties of Busemann
spaces (convexity of balls, local geodesics are geodesics) in
Subsection \ref{subsection:auxiliary}.  Our proofs require some other
properties of convexity and distance function in Busemann surfaces,
which will be established in Subsection \ref{subsection:auxiliary}:
monotonicity of perimeters of triangles, convexity preserves diameters
of sets, Helly theorem, convexity of shades of geodesic segments and
of triangles, line-separation of a triangle and a point not belonging
to this triangle, to mention some of them.

\subsection{Covering and packing with balls}

Let $(X,d)$ be a metric space, $S$ be a subset of $X$, and 
${\delta}$ be an arbitrary positive real number. For a point 
$x\in X$, we will denote by
$B_{\delta}(x)=\{ y\in X: d(x,y)\le {\delta}\}$ and
$B_{\delta}^{\circ}(x)=\{ y\in X: d(x,y)< {\delta}\}$ the {\it closed}
and the {\it open balls} of radius ${\delta}$ and center $x$.  A {\it
${\delta}$-simplex} is a subset $Y$ of $X$ of diameter at most
$2{\delta}$, i.e., $d(x,y)\le 2{\delta}$ for any $x,y\in Y$  The {\it Rips} 
(or the {\it Vietoris-Rips}) {\it complex} $P_{\delta}(S)$ of $S$
\cite[p.468]{BrHa} is a simplicial complex whose vertices are the
points of $S$ and a subset $Y\subseteq S$ is a simplex of
$P_{\delta}(S)$ if and only if diam$(Y)\le {\delta}$, i.e., if $Y$ is
a $\frac{{\delta}}{2}$-simplex. Denote by $G_{\delta}(S)$ the
1-skeleton of $P_{\delta}(S)$, i.e., $S$ is the vertex-set of
$G_{\delta}(S)$ and $x,y$ are adjacent in $G_{\delta}(S)$ if and only
if the pair $x,y$ defines a simplex of $P_{\delta}(S)$, i.e.,
$d(x,y)\le {\delta}$. Notice that $P_{\delta}(S)$ is the clique
complex of $G_{\delta}(S)$. Finally, let $\overline{G}_{\delta}(S)$ 
denote the complement of the graph $G_{\delta}(S)$. 

For a given radius ${\delta}>0$, a set of
closed balls ${\mathcal C}=\{ B_{{\delta}}(x_i): i\in I\}$ with
centers $x_i\in X$ is called a {\it covering} of a set $S$ if $S\subseteq
\bigcup_{i\in I} B_{\delta}(x_i)$. Analogously, a set of open balls
${\mathcal C}^{\circ}=\{ B^{\circ}_{{\delta}}(x_i): i\in I\}$ is
called an {\it open covering} of $S$ if $S\subseteq \bigcup_{i\in I}
B^{\circ}_{\delta}(x_i)$. Denote by $\rho_{\delta}(S)$ (respectively,
by $\rho^{\circ}_{\delta}(S)$) the minimum number of balls of radius
${\delta}$ in a covering (respectively, in a open covering) of $S$,
and call $\rho_{\delta}(S)$ and $\rho^{\circ}_{\delta}(S)$ the {\it
covering} and the {\it open covering numbers} of $S$. Obviously,
$\rho_{\delta}(S)\le \rho^{\circ}_{\delta}(S)$. If $S$ is compact,
then $\rho^{\circ}_{\delta}(S)$ is finite, and therefore
$\rho_{\delta}(S)$ is finite as well.

A set of closed balls ${\mathcal P}=\{ B_{{\delta}}(x_i): i\in I\}$
with centers $x_i\in S$ is called a {\it packing} of $S\subseteq X$ if
the balls of ${\mathcal P}$ are pairwise disjoint. Analogously, a set
of open balls ${\mathcal P}^{\circ}=\{ B_{{\delta}}^{\circ}(x_i): i\in
I\}$ with centers $x_i\in S$ is called an {\it open packing} of $S$ if
the balls of ${\mathcal P}^{\circ}$ are pairwise disjoint. Denote by
$\nu_{\delta}(S)$ the maximum number of closed balls in a packing of
$S$, i.e., the size of a largest subset $P$ of $S$ such that
$d(x_i,x_j)>2{\delta}$ for any two distinct points $x_i,x_j$ of $P$,
and call $\nu_{\delta}(S)$ the {\it packing number} of
$S$. Analogously, the {\it open packing number}
$\nu^{\circ}_{\delta}(S)$ is the size of a largest subset $P$ of $S$
such that $d(x_i,x_j)\ge 2{\delta}$ for any two distinct points
$x_i,x_j$ of $P$. Clearly, for any $S\subseteq X$, the following
inequalities hold: $\nu_{\delta}(S)\le \nu^{\circ}_{\delta}(S)$,
$\nu_{\delta}(S)\le \rho_{\delta}(S)$, and $\nu^{\circ}_{\delta}(S)\le
\rho^{\circ}_{\delta}(S)$.  Therefore, if $S$ is compact, then
$\nu(S)$ and $\nu^{\circ}_{\delta}(S)$ are finite as well.
Finally, a {\it
${\delta}$-simplex covering} of $S$ is a collection $\mathcal R=\{
Y_i: i\in I\}$ of ${\delta}$-simplices such that $Y_i\subseteq S$ and
$S=\bigcup_{i\in I} Y_i$. The {\it ${\delta}$-simplex covering number}
$\theta_{\delta}(S)$ of $S$ is the minimum number of
${\delta}$-simplices in a covering of $S$. Notice that
$\theta_{\delta}(S)=1$ (i.e., $S$ is an ${\delta}$-simplex) if and
only if $\nu_{\delta}(S)=1$.

We will say that a class $\mathcal M$ of metric spaces has the {\it
bounded covering-packing property} if there exists a universal
constant $c$ such that for any metric space $(X,d)$ from ${\mathcal
M}$, any ${\delta}>0$, and any compact subset $S$ of $X$, the
inequality $\rho_{\delta}(S)\le c\nu_{\delta}(S)$ holds. We will also
say that $\mathcal M$ has the {\it bounded simplex-ball covering
property}, if there exists a universal constant $c$ such that for any
$(X,d)\in {\mathcal M}$ and any ${\delta}>0$, any ${\delta}$-simplex
$S$ of $X$ can be covered by at most $c$ balls of radius
${\delta}$. Recall also that a class $\mathcal G$ of graphs is {\it
linearly $\chi$-bounded} if there exists a constant $c$ such that
$\chi(G)\le c\omega(G)$ for any graph $G\in {\mathcal G}$.

\begin{lemma} \label{bounded-cover}
Let $\mathcal M$ be a class of metric spaces having the bounded
simplex-ball covering property. If the class of graphs ${\mathcal
G}=\{ \overline{G}_{2{\delta}}(S): {\delta}>0 \mbox{ and } S \mbox{ is
a compact subset of } X\}$ is linearly $\chi$-bounded, then $\mathcal
M$ satisfies the bounded covering-packing property.
\end{lemma}

\begin{proof}
  Since any coloring of $\overline{G}_{2{\delta}}(S)$ is a clique
  covering of ${G}_{2{\delta}}(S)$ and each clique of
  ${G}_{2{\delta}}(S)$ is a $\delta$-simplex of $S$, the set $S$
  admits a $\delta$-simplex covering with at most
  $c\omega(\overline{G}_{2{\delta}}(S))$ simplices.  If $(X,d)$ has
  the bounded covering-packing property with constant $c'$, we
  conclude that $S$ can be covered with at most $c'c
  \omega(\overline{G}_{2{\delta}}(S))=c'c\nu_{\delta}(S)$ balls of
  radius $\delta$.
\end{proof}

An important class of metric spaces satisfying the bounded
covering-packing property (and extending the Euclidean spaces) is
constituted by metric spaces with {\it bounded doubling dimension},
i.e., metric spaces $(X,d)$ in which for any ${\delta}>0$ any ball of
radius $2{\delta}$ of $X$ can be covered with a constant number of
balls of radius ${\delta}$~\cite{Cl}. We will relax this doubling
property in the following way. We will say that a metric space $(X,d)$
satisfies the {\it weak doubling property} if there exists a constant
$c$ such that for any ${\delta}>0$ and any compact set $S\subseteq X$,
there exists a point $v\in S$ such that $B_{2{\delta}}(v)\cap S$ can
be covered with at most $c$ balls of radius ${\delta}$ of $X$. The
proof of the following result will be given in the next section:

\begin{proposition} \label{weak-doubling}
If a complete metric space $(X,d)$ satisfies the weak doubling
property with constant $c$, then for any compact set $S\subseteq X$
and any ${\delta}>0$, $\rho_{\delta}(S)\le c\nu_{\delta}(S)$.
\end{proposition}

\subsection{Related work}

Kolmogorov and Tikhomirov~\cite{KoTi} introduced the three covering
and packing numbers (under different notations and names) and noticed
the following simple but fundamental relationship between them: for
any completely bounded (in particular, compact) subset $S$ of an
arbitrary metric space $(X,d)$,
$$\nu_{{\delta}}(S)\le \theta_{\delta}(S)\le \rho_{\delta}(S)\le
\nu_{\frac{{\delta}}{2}}(S).$$ Furthermore, they called the binary
logarithms of the quantities $\theta_{\delta}(S), \rho_{\delta}(S),$
and $\nu_{\delta}(S)$ the {\it ${\delta}$-entropy} of $S$, the {\it
${\delta}$-entropy} of $S$ with respect to $X$, and the {\it
${\delta}$-capacity} of $S$, respectively (also called {\it metric
entropy} and {\it metric capacity} of $S$). These quantities found
numerous applications in pure and applied mathematics~\cite{Lo},
probability theory and statistics~\cite{Du}, learning theory
~\cite{KuMi}, and computational geometry~\cite{Cl}, just to name some.

Notice also the following graph-theoretical interpretation of covering
and packing numbers $\theta_{\delta}(S),\nu_{\delta}(S),$ and
$\rho_{\delta}(S)$. A ${\delta}$-simplex covering of $S$ in the sense of Kolmogorov
and Tikhomirov corresponds to a covering of $S$ by simplices of the 
Rips complex $P_{2{\delta}}(S)$ and to a clique cover of $G_{2{\delta}}(S)$;
therefore $\theta_{\delta}(S)$ corresponds to the size of a minimum
clique covering of $G_{2{\delta}}(S)$, i.e., to the chromatic number
$\chi(\overline{G}_{2{\delta}}(S))$ of the complement
$\overline{G}_{2{\delta}}(S)$ of the graph
$G_{2{\delta}}(S)$. Analogously, a packing of $S$ corresponds to a
stable set of $G_{2{\delta}}(S)$, i.e., to a clique of
$\overline{G}_{2{\delta}}(S)$; consequently, $\nu_{\delta}(S)$ equals
the clique number $\omega(\overline{G}_{2{\delta}}(S))$ of the
complement of $G_{2{\delta}}(S)$. Finally, $\rho_{\delta}(S)$
corresponds to the domination number of $G_{\delta}(S)$, i.e., to the
minimum covering of $S$ by stars of $G_{\delta}(S)$.

It was shown in~\cite{ChEsVa} that the class ${\mathcal M}_{\sf
planar}$ of all metric spaces obtained as standard graph-metrics of
planar graphs has the bounded simplex-ball covering property. In
~\cite{BoCh}, this result was generalized to all graphs on surfaces of
a given genus; see also~\cite{Bou,BouTho} for other generalizations of
the result of~\cite{ChEsVa}. It was conjectured in \cite[Problem
5]{Ca} that the class ${\mathcal M}_{\sf planar}$ has the bounded
covering-packing property, namely, that it satisfies the weak doubling
property. Notice also, that it was shown in~\cite{ChEs} that if $S$ is
a compact subset of a geodesic $\varepsilon$-hyperbolic space (in the
sense of Gromov) or of an $\varepsilon$-hyperbolic graph, then
$\rho_{{\delta}+2\varepsilon}(S)\le \nu_{\delta}(S)$ (compare it with the
general inequality $\nu_{{\delta}}(S)\le \rho_{\delta}(S)\le
\nu_{\frac{{\delta}}{2}}(S)$). This result can be interesting if the
hyperbolicity $\varepsilon$ constant is much smaller than the radius
$\delta$ of balls used in the covering.

There exists a strong analogy between the properties of graphs and
geodesic metric spaces, due to their uniform local structure. Any
graph $G=(V,E)$ gives rise to a network-like geodesic space (into
which $G$ isometrically embeds) obtained by replacing each edge $xy$
of $G$ by a segment isometric to $[0,1]$ with ends at $x$ and
$y$. Conversely, by \cite[Proposition 8.45]{BrHa}, any geodesic metric
space $(X,d)$ is (3,1)-quasi-isometric to a graph $G=(V,E)$. (Let $(X_1,d_1)$
and $(X_2,d_2)$ be metric spaces. A map $f:X_1\rightarrow X_2$ is called a
$(\lambda,\epsilon)$-{\it quasi-isometric embedding} if there exists constants
$\lambda\ge 1$ and $\epsilon\ge 0$ such that for all $x,y\in X_1$,
$\frac{1}{\lambda}d_1(x,y)-\epsilon\le d_2(f(x),f(y))\le \lambda d_1(x,y)+\epsilon$.) This
graph $G$ is constructed in the following way: let $V$ be an open
$\frac{1}{3}$-packing of $X$ (it exists by Zorn's lemma but can be
infinite). Then two points $x,y\in V$ are adjacent in $G$ if and only
if $d(x,y)\le 1$.

Due to this analogy, one can formulate the previous question about
${\mathcal M}_{\sf planar}$ for their continuous counterparts
${\mathcal M}_{\sf polygon}$--- polygons in ${\mathbb R}^2$ endowed
with the (intrinsic) geodesic metric.  It turns out that this question
was not yet considered even for simple polygons (in this case, only a
factor 2 approximation algorithm for packing number was recently given
in~\cite{Vi}). The geodesic metric on simple polygons was studied in
several papers in connection with algorithmic problems. In particular,
in was shown in~\cite{PoShRo}, that balls are convex, implying that
simple polygons are Busemann spaces.  In this paper, we consider the
relationship between the packing and covering numbers not only for
simple polygons in the Euclidean or hyperbolic planes but also for
(compact subsets of) general Busemann surfaces.

\subsection{The main results}

We continue with statements of the main results of this
note. Starting from now, we will denote $\rho_{\delta}(S)$ and
$\nu_{\delta}(S)$ by $\rho(S)$ and $\nu(S)$, respectively.

\begin{theorem} \label{packing_covering}
 Let $S$ be a compact subset of a Busemann surface $({\mathcal S},d)$
and ${\delta}$ an arbitrary positive number. Then $\rho(S)\le
\constant{}\nu(S)$.
\end{theorem}

\begin{corollary} \label{cor:polygon}
Let $\mathcal P$ be a simple polygon in ${\mathbb R}^2$.  Then
$\nu({\mathcal P})\le \rho({\mathcal P})\le \constant{}\nu({\mathcal P})$ for
any $\delta>0$.
\end{corollary}

\begin{proof}
Let $\mathcal P$ be a simple polygon endowed with the geodesic
metric. In~\cite{ChChNa} it was shown how to extend $\mathcal
P$ to a Busemann surface $({\mathcal S},d)$. Notice that by
this construction, $\mathcal P$ is embedded as a convex subset of
${\mathcal S}$. Since $\mathcal P$ is a compact subset of $\mathcal S$,
$\rho({\mathcal P})\le \constant{}\nu({\mathcal P})$ by Theorem~\ref{packing_covering}.
Let ${\mathcal C}=\{ B_{\delta}(x_1),\ldots, B_{\delta}(x_k)\}$
be a covering of $\mathcal P$ with closed ${\delta}$-balls of
$({\mathcal S},d)$ constructed as in the proof of Propositions
\ref{prop:three-balls} and \ref{prop:twenty-three-balls}. Since
$\mathcal P$ is a compact convex subset of $\mathcal S$,
the centers of the balls of $\mathcal C$ will belong to $\mathcal P$,
concluding the proof of Corollary ~\ref{cor:polygon}.
\end{proof}

The proof of Theorem~\ref{packing_covering} immediately follows from
Proposition~\ref{weak-doubling} and Proposition
\ref{prop:twenty-three-balls} formulated below and which establishes
that Busemann surfaces satisfy the weak doubling property.  One
essential ingredient in the proof of Proposition
\ref{prop:twenty-three-balls} is the bounded simplex-ball covering
property established in Proposition \ref{prop:three-balls}.  We
continue with the precise formulation of these two results.

Proposition \ref{prop:three-balls} extends the well-known folkloric
result by Hadwiger and Debrunner~\cite{HaDe} that any set of pairwise
intersecting unit balls in the plane can be pierced by three needles
(answering a question by Gr\"unbaum, this result was extended in
\cite{Ka} to translates of any convex compact set of
${\mathbb R}^2$). Namely, we show that Busemann surfaces satisfy the
bounded simplex-ball covering property with constant 3:

\begin{proposition} \label{prop:three-balls}
Let $S$ be a compact subset of a Busemann surface $({\mathcal S},d)$
and suppose that the diameter of $S$ is at most $2{\delta}$. Then $S$
can be covered with 3 balls of radius ${\delta}$, i.e., $\rho(S)\le
3$.
\end{proposition}

The second result shows that Busemann surfaces satisfy the weak
doubling property:

\begin{proposition} \label{prop:twenty-three-balls}
 Let $S$ be a compact subset of a Busemann surface $({\mathcal S},d)$
and let $u,v\in S$ be a diametral pair of $S$. Then $B_{2{\delta}}(v)\cap S$
can be covered by \constant{} balls
of radius ${\delta}$.
\end{proposition}

The idea of proof of Proposition~\ref{prop:twenty-three-balls} is to
partition the set $B_{2{\delta}}(v)\cap S$ into six regions, four of
them of diameter $\le 2{\delta}$ and to which we can apply Proposition
~\ref{prop:three-balls} and two regions which can be covered with
eight balls.

\begin{remark}
Notice that Busemann surfaces (unlike
Euclidean and hyperbolic planes) do not have bounded doubling dimension,
i.e., not every ball $B_{2\delta}(v)$ of radius $2\delta$ can be
covered with a fixed number of balls of radius $\delta$. Indeed, for
any positive integer $n$, the star $S_n$ with $n$ leaves $u_1,\ldots
u_n$, center $v$, and length $2\delta$ of all edges can be embedded
isometrically into a Busemann surface ${\mathcal S}_n$ in the
following way. First embed $S_n$ into a star $\hat{S}_n$ consisting of
$n$ rays $R_i$, $i=1,\ldots,n$, with center $v$, where $R_i$ is the
ray passing via the leaf $u_i$ of $S_n$. Notice that the union
$L_{i,j}$ of any two distinct rays $R_i$ and $R_j$ is isomorphic to the
real line ${\mathbb R}$.  To each line $L_{i,i+1},$ $i=1,\ldots n$
(where $i+1$ is taken modulo $n$), of $\hat{S}_n$ we add a closed
halfplane $H_{i,i+1}$ of ${\mathbb R}^2$ so that $L_{i,i+1}=R_i\cup
R_{i+1}$ is the boundary of $H_{i,i+1}$. Two consecutive halfplanes
$H_{i-1,i}$ and $H_{i,i+1}$ intersect in the common ray $R_i$. Two
nonconsecutive halfplanes intersect only in the center $v$ of
$S_n$. Let ${\mathcal S}_n$ be the planar surface obtained as the
union of the $n$ closed halfplanes $H_{i,i+1}$, $i=1,\ldots n$.  It can
be easily seen that the resulting planar surface ${\mathcal S}_n$ is
Busemann (if fact, it is CAT(0)) and that $S_n$ and $\hat{S}_n$ are
isometrically embedded into ${\mathcal S}_n$. Now, consider the ball
$B_{2\delta}(v)$ of ${\mathcal S}_n$ centered at the center $v$ of
$S_n$. Since the distance from $v$ to any of the leaves $u_i$ of $S_n$
in $S_n$ and ${\mathcal S}_n$ is $2\delta$, $\{ u_1,\ldots,
u_n\}\subset B_{2\delta}(v)$. On the other hand, since the distance in
$S_n$ and ${\mathcal S}_n$ between any two different leaves $u_i$ and
$u_j$ is $4\delta$, any covering in ${\mathcal S}_n$ of the set $\{
u_1,\ldots, u_n\}$ with balls of radius $\delta$ requires at least $n$
balls. Consequently, any covering of $B_{2\delta}(v)$ with balls of
radius $\delta$ requires at least $n$ balls.
\end{remark}

\section{Proofs}

In this section, we provide the proofs of Propositions 1-3.
We start with the proof of Proposition \ref{weak-doubling}, presented
in Subsection \ref{subsection:weak-dubling}. The proofs of
Propositions~\ref{prop:three-balls} and \ref{prop:twenty-three-balls}
require some geometric properties of Busemann surfaces, which we
present in Subsection~\ref{subsection:auxiliary}. The proof of
Proposition~\ref{prop:three-balls} is presented in
Subsection~\ref{subsection:three-balls} and the proof of
Proposition~\ref{prop:twenty-three-balls} is given in
  Subsection~\ref{subsection:twenty-three-balls}.

\subsection{Proof of Proposition~\ref{weak-doubling}}\label{subsection:weak-dubling}

In this subsection, we will prove Proposition~\ref{weak-doubling}, which we recall now:

\medskip\noindent
{\bf Proposition 1}. {\it If a complete metric space $(X,d)$ satisfies the weak doubling
property with constant $c$, then for any compact set $S\subseteq X$
and any ${\delta}>0$, $\rho_{\delta}(S)\le c\nu_{\delta}(S)$.
}

\begin{proof}
The proof of Proposition~\ref{weak-doubling} is algorithmic and builds
simultaneously (in a primal-dual way) a covering $\mathcal C$ of $S$
with closed ${\delta}$-balls and an open packing $P$ of $S$ satisfying
the inequality $|\mathcal C|\le c|P|$.  Since $P$ is an open packing
and $S$ is compact, $|P|\le \nu^{\circ}(S)\le \rho^{\circ}(S)<\infty$,
thus $P$ and $\mathcal C$ are finite and their construction requires a
finite number of steps.  Then using local perturbations, we will show
how to transform $P$ into a packing $P'$ of the same size as $P$.

Start by setting $S^*_0:=S$, $S_0:=S$, $\mathcal C:=\varnothing$,
$P:=\varnothing$, and $i=0$. While $S_i\ne \varnothing$, set
$S^*_i:=\overline{S}_i$ (the closure of $S_i$). Since $(X,d)$ is
complete, $S^*_i$ is compact. Since $(X,d)$ satisfies the weak
doubling property, $S^*_i$ contains a point $v$ such that the set
$B_{2{\delta}}(v)\cap S^*_i$ can be covered with $k\le c$ balls
$B_{\delta}(x_1),\ldots, B_{\delta}(x_k)$ of radius ${\delta}$ of
$X$. Add the balls $B_{\delta}(x_1),\ldots, B_{\delta}(x_k)$ to the
covering $\mathcal C$, denote the point $v$ by $p_i$ and add it to
$P$. Finally, set
$S_{i+1}:=S_i\setminus(\bigcup_{j=1}^kB_{\delta}(x_j))$ and
$S^*_{i+1}:=\overline{S}_{i+1}$, and apply the algorithm to these two
new sets.

We claim that $P$ is an open packing of $S$. Pick any pair of points
$p_i,p_j\in P$ and let $j<i$. Then $p_i$ is either a point of $S_i$ or
$p_i$ is the limit of an infinite sequence $\{ s_t\}$ of points of
$S_i$. From its definition, the set $S_i$ consists of all yet not
covered by $\mathcal C$ points of $S$; in particular, we have $S_i\cap
(\bigcup_{k=1}^{i-1}B_{2{\delta}}(p_k))=\varnothing$. Consequently, if
$p_i\in S_i$, since $p_i\notin B_{2{\delta}}(p_j)$, we conclude that
$d(p_i,p_j)>2{\delta}$ in this case. Now, suppose that $p_i$ is the
limit of a sequence $\{ s_t\}$ of points of $S_i$. If
$d(p_i,p_j)<2{\delta}$, then for any $\varepsilon>0$ such that
$d(p_i,p_j)+\varepsilon<2{\delta}$, all points of $\{ s_t\}$ except a
finite number will be in the $\varepsilon$-neighborhood of $p_i$. For any
such point $s_t$, we will have $d(s_t,p_j)\le d(s_t,p_i)+d(p_i,p_j)\le
\varepsilon+d(p_i,p_j)<2{\delta}$, contrary to the choice of $s_t$ from
$S_i$. This contradiction shows that $P$ is an open packing of
$S$. Consequently, $P$ and $\mathcal C$ are finite, and from their
construction, $|\mathcal C|\le c|P|$.

Now, we will show how to transform the finite open packing $P=\{
p_1,\ldots, p_n\}$ of $S$ into a packing $P'$ of the same size. For
this we will move each point of $P$ at most once. We proceed the
points of $P$ in the reverse order and for each point $p_i$ of $P$
either we include it in $P'$ (and denote it by $p'_i$) or include in
$P'$ a point $p'_i\in S_i$. Suppose that after proceeding the points
$p_n,\ldots,p_{i+1},$ the set $P'$ has the form $P'=\{
p_1,\ldots,p_i,p'_{i+1},\ldots, p'_n\}$ and satisfies the following
invariants: (a) $d(p_j,p'_k)>2{\delta}$ for any $j=1,\ldots,i$ and
$k=i+1,\ldots,n$ and (b) $d(p'_j,p'_k)>2{\delta}$ for any $i+1\le
j<k\le n$.  We will show how to proceed the point $p_i$ to keep valid
the invariants (a) and (b). If $d(p_i,p_j)>2{\delta}$ for any $j<i$,
then we simply set $p_i'=p_i$ and obviously (a) and (b) are
preserved. Otherwise, suppose that there exists a point $p_j$ with
$j<i$ such that $d(p_i,p_j)=2{\delta}$. By the construction of $P$ and
the argument in the proof that $P$ is an open packing, we conclude
that $p_i\notin S_i$ and therefore $p_i$ is a limit of an infinite
sequence $\{ s_t\}$ of points of $S_i$. In the basis case $i=n$ we
simply pick as $p'_n$ any point from the sequence $\{
s_t\}$. Obviously, the conditions (a) and (b) will be preserved. Now,
suppose that $i<n$.  Let $\varepsilon:=\min \{ d(p_i,p'_k) - 2 \delta :
k>i\}$.  Clearly, $\varepsilon>0$. Pick as $p'_i$ any point of the
sequence $\{ s_t\}$ lying in the $\frac{\varepsilon}{2}$-neighborhood of
$p_i$. Then $d(p_j,p'_i) > 2 \delta$ for any $j < i$, because $p'_i\in
S_i$. Also $d(p'_i,p'_k) > 2 \delta$ for any $k > i$ because by triangle
inequality $d(p'_i,p'_k) > d(p_i,p'_k) - d(p'_i,p_i) > d(p_i,p'_k) -
\frac{\varepsilon}{2} > 2{\delta}$. This shows that after proceeding all
points of $P$, we will obtain a set $P'$ of $n$ points of $S$,
satisfying the conditions (a) and (b), i.e., a packing of $S$. This
finishes the proof of Proposition~\ref{weak-doubling}.
\end{proof}

\subsection{Auxiliary results}\label{subsection:auxiliary}

In this subsection, we present some elementary properties of Busemann
planar surfaces. We start with some fundamental properties of all Busemann spaces.

\begin{lemma} \label{unique_geodesics} \cite[Proposition 8.1.4]{Pa} A Busemann space is uniquely geodesic.
\end{lemma}

\begin{lemma}\label{local} \cite[Corollary 8.2.3]{Pa}
Every local geodesic of a Busemann space $(X,d)$ is a geodesic.
\end{lemma}

From these two lemmas immediately follows that geodesic lines of
Busemann spaces do not self-intersect.

Let $(X,d)$ be a Busemann space. For two points $x,y$ of $X$, we
denote by $[x,y]$ the unique geodesic segment joining $x$ and $y$. We
will also denote a line containing $x$ and $y$ by $(x,y)$ when there
is no ambiguity (there may be many such lines). A set
$R\subseteq {\mathcal S}$ is called {\it convex} if
$[p,q] \subseteq R$ for any $p,q \in R$.  For a set $Q$ of
$\mathcal S$ the smallest convex set conv$(Q)$ containing $Q$ is
called the {\it convex hull} of $Q$.  The next lemma immediately
follows from the definition of Busemann spaces.

\begin{lemma} \label{ball-convexity} \cite[Proposition 8.3.1]{Pa}
The open balls and closed balls  of a Busemann space $(X,d)$  are convex.
\end{lemma}

A geodesic metric space $(X,d)$ is said to have the {\it geodesic
extension property} if the geodesic $[x,y]$ between any two
distinct points $x,y$ can be extended to a {\it geodesic line}, i.e., to a
line $(x,y)$ passing via $x$ and $y$.
Based on \cite[Footnote 24]{BrHa},  it was noticed in \cite[Lemma 1]{ChChNa}
that Busemann spaces have the extension property:

\begin{lemma} \label{extension}
Any Busemann surface ${\mathcal S}$ has the geodesic extension property.
\end{lemma}

From now suppose that $({\mathcal S},d)$ is a Busemann surface.
For a geodesic line $\ell$, we denote by $H'_{\ell}$ and $H''_{\ell}$
the unions of the two connected components of ${\mathcal S}\setminus
\ell$ with $\ell$. We call $H'_{\ell}$ and $H''_{\ell}$ {\it closed
halfplanes}. Since each line is convex, $H'_{\ell}$ and $H''_{\ell}$
are convex sets of $\mathcal S$. We will say that a line $\ell$ {\it separates}
two sets $A$ and $B$ if $A$ and $B$ belong to different closed halfplanes
defined by $\ell$.

For three points $x,y,z$ of $\mathcal S$, the {\it geodesic triangle}
$\partial\Delta(x,y,z)$ is the union of the three geodesics
$[x,y],[y,z],$ and $[z,x]$. We will call the closed bounded region
$\Delta(x,y,z)$ of $\mathcal S$ bounded by $\partial\Delta(x,y,z)$ the
{\it triangle} with vertices $x,y,z$.  We will say that the triangle
$\Delta(x,y,z)$ is {\it degenerated} if the points $x,y,z$ are
collinear, i.e., one of these points belongs to the geodesic between
the other two. By a {\it (convex) quadrangle} we will
mean the convex hull of four point $x,y,z,v$ in convex position, i.e.,
neither of the four points is in the convex hull of the other
three. For two distinct points $u,x \in {\mathcal S}$, let
$\Sp_u(x):=\{ p \in {\mathcal S}: x \in [u,p]\}$; we will call the set
$\Sp_u(x)$ a {\it cone}. Since $\mathcal S$ satisfies the geodesic
extension property, the set $\Sp_y(x) \cup [x,y] \cup \Sp_x(y)$ can be
equivalently defined as the union of all geodesic lines extending
$[x,y]$.

We continue by recalling some results from \cite{ChChNa}.
We start with a Pasch axiom, which we formulate in a slightly stronger
but equivalent form:

\begin{lemma}\cite[Lemma 6]{ChChNa}  \label{Pasch} (Pasch axiom)
  If $\Delta(x,y,z)$ is a triangle, $u \in [x,y], v\in [x,z]$, and $p \in
  [y,z],$ then $[u,v] \cap [x,p] \ne \varnothing$.
\end{lemma}

\begin{lemma}  \cite[Lemma 7]{ChChNa} \label{cones}
  The cone $\Sp_u(x)$ is a convex and closed subset of $\mathcal S$.
\end{lemma}

\begin{lemma} \cite[Lemma 8]{ChChNa} \label{triangle}
  $\Delta(x,y,z)$ coincides with the convex hull of $x,y,z$.
\end{lemma}

\begin{lemma}  \cite[Lemma 9]{ChChNa} \label{Peano} (Peano axiom)
  If $\Delta(x,y,z)$ is a triangle, $p \in [x,y]$, $q \in
  [x,z]$, and $u \in [p,q],$ then there exists a point $v \in [y,z]$ such
  that $u \in [x,v]$.
\end{lemma}

The next lemma asserts that the rays of two tangent lines at a point
$x$ induce one or two additional lines in their support (for an
illustration, see Fig. 1 of \cite{ChChNa}):

\begin{lemma} \cite[Lemma 5]{ChChNa}\label{two-lines} Let $\ell$ and
  $\ell'$ be two intersecting geodesic lines such that $\ell'$ is
  contained in a closed halfplane $H$ defined by $\ell$. Let
  $x \in \ell \cap \ell'$, and let $r_1,\ldots r_4$ be the four rays
  emanating from $x$ with $\ell = r_1 \cup r_2$ and
  $\ell' = r_3 \cup r_4$ and $r_1,r_4,r_3,r_2$ appear in that order
  around $x$ on $H$. Then $r_1 \cup r_3$ and $r_2 \cup r_4$ are also
  geodesic lines.
\end{lemma}

Since a Busemann surface $\mathcal S$ is homeomorphic to the plane
${\mathbb R}^2$, the properties of ${\mathbb R}^2$ preserved by
homeomorphisms also hold in $\mathcal S$.  For example, any simple
closed curve $\gamma$ in $\mathcal S$ divides the surface $\mathcal S$
into an interior region ${\mathcal R}:={\mathcal R}(\gamma)$ bounded
by $\gamma$ and an exterior region. Moreover, ${\mathcal R}$ is a
contractible bounded subset of $\mathcal S$.  A {\it cut} of
${\mathcal R}$ with endpoints $x,y\in \gamma$ is a path $\mu:
[a,b]\rightarrow {\mathcal R}$ such that $\mu (a)=x, \mu(b)=y$, and
$\mu(c)\in {\mathcal R}$ for any $a\le c\le b$. Using the
homeomorphism between $\mathcal S$ and ${\mathbb R}^2$, one can see
that any cut $\mu$ of $\mathcal R$ divides $\mathcal R$ into two
contractible bounded regions. Analogously, if $x,u,y,v$ are four
points occurring in this order on $\gamma$,  $\mu'$ is a cut of
$\mathcal R$ with endpoints $x,y$, and $\mu''$ is a cut of $\mathcal R$
with endpoints $u,v$, then $\mu'$ and $\mu''$ cross and divide
$\mathcal R$ into four contractible regions.

Using this kind of arguments, one can derive the following basic
properties of Busemann surfaces:
\begin{itemize}
\item[(1)] If $\Delta(x,y,z)$ is a triangle and $t \in [y,z]$, then
$\Delta(x,y,z)$ is divided into two triangles $\Delta(x,y,t)$ and
$\Delta(x,z,t)$ (i.e., $\Delta(x,y,z) = \Delta(x,y,t) \cup \Delta(x,t,z)$ and
$\Delta(x,y,t) \cap \Delta(x,t,z) = [x,t]$);
\item[(2)] If $\Delta(x,y,z)$ is a triangle and $u\in [x,y], v\in
[x,z],$ and $w\in [y,z],$ then $\Delta(x,y,z)$ is divided into four
triangles $\Delta(x,u,v),\Delta(v,w,z),\Delta(u,w,y)$, and
$\Delta(u,v,w)$;
\item[(3)] If $\Delta(x,y,z)$ is a triangle and $u\in
\Delta(x,y,z)$, then $\Delta(x,y,z)$ is divided into three
triangles $\Delta(x,y,u),\Delta(y,z,u),$ and $\Delta(x,z,u)$;
\item[(4)] If $Q=\mbox{ conv}(x,y,z,u)$ is a convex quadrangle with
sides $[x,y],[y,z],[z,u],[u,x]$ and $p\in [x,y], s\in [y,z], q\in
[z,u], t\in [u,x]$, then the geodesic segments $[p,q]$ and $[s,t]$
divide $Q$ into four convex quadrangles.
\end{itemize}

We will denote by $\partial B_r(x)$ the {\it sphere} of center $x$ and
radius $r$;  $\partial B_r(x)$ can be viewed as the difference between
$B_r(x)$ and $B_r^{\circ}(x)$ or, equivalently, as the set $\{ y\in
{\mathcal S}: d(x,y)=r \}$. The following property is also a
consequence of the homeomorphism between $\mathcal S$ and ${\mathbb
R}^2$:

\begin{lemma} \label{sphere1}
Any sphere $\partial B_r(x)$ of $\mathcal S$ is homeomorphic to the
circle ${\mathbb S}^1$ of ${\mathbb R}^2$.
\end{lemma}

We continue with some new properties of Busemann surfaces.
Let $\pi(x,y,z)$ denote the perimeter of $\Delta(x,y,z)$, i.e.,
$\pi(x,y,z)=d(x,y)+d(y,z)+d(z,x)$. Then the following
monotonicity properties of triangles holds:

\begin{lemma} \label{perimeter}
  If $x',y',z'\in \Delta(x,y,z),$ then
  $\pi(x',y',z')\le \pi(x,y,z)$. Moreover, the equality holds only if
  either $\{ x',y',z'\}=\{ x,y,z\}$ or $\Delta(x,y,z)$ is degenerated,
  i.e., the points $x,y,z$ are collinear.
\end{lemma}

\begin{proof}
  First assume that $\{ x',y',z'\} \subset \partial\Delta(x,y,z)$.
  Then the inequality $\pi(x',y',z')\le \pi(x,y,z)$ easily follows by
  applying the triangle inequality.

  Otherwise we may assume by symmetry that
  $x' \notin \partial\Delta(x,y,z)$. By Lemma~\ref{triangle}
  (convexity of triangles), $(x',y') \cap \Delta(x,y,z)$ is a segment
  $[x'',y'']$ with $x'', y'' \in \partial\Delta(x,y,z)$. Again by
  convexity of triangles, $(x'',z') \cap \Delta(x,y,z)$ is a segment
  $[x'',z'']$ with $x'', z'' \in \partial\Delta(x,y,z)$ and such that
  $z' \in [x'',z'']$. Since $x'', y'', z'' \in \partial\Delta(x,y,z)$,
  by the first case we have $\pi(x'',y'',z'') \leq \pi(x,y,z)$. By
  construction, $x', y' \in [x'',y'']$ and $z' \in [x'',z'']$, whence
  again by the first case we have
  $\pi(x',y',z') \le \pi(x'',y'',z'')$. Consequently,
  $\pi(x',y',z') \le \pi(x,y,z)$.

  The case of equality follows easily in the first case and from the
  fact that we reduced the general case to the first case.
\end{proof}

\begin{lemma} \label{perimeter2}
  If $u,v\in \Delta(x,y,z)$ and $d(x,y),d(y,z),d(z,x)\le \delta$, then
  $d(u,v)\le \delta$.
\end{lemma}

\begin{proof}
  Since $x, y, z \in B_{\delta}(x)$ and the ball $B_{\delta}(x)$ is
  convex, $\Delta(x,y,z) \subseteq B_{\delta}(x)$. Hence
  $u \in B_{\delta}(x) \cap B_{\delta}(y) \cap B_{\delta}(z)$, or
  equivalently $x, y, z \in B_{\delta}(u)$. Again, since
  $B_{\delta}(u)$ is convex,
  $v \in \Delta(x,y,z)\subseteq B_{\delta}(u)$, whence
  $d(u,v)\le \delta$.
\end{proof}

We continue with the following quadrangle condition:

\begin{lemma} \label{quadrangle}
If $x,y,u,v$ are four points of $\mathcal S$ such that $[x,y]\cap
[u,v]\ne\varnothing$, then $\max\{ d(x,u)+d(y,v), d(x,v)+d(y,u)\}\le
d(x,y)+d(u,v)$. \end{lemma}

\begin{proof}
Let $z\in [x,y]\cap [u,v]$. By triangle inequality, $d(x,u)\le
d(x,z)+d(z,u)$ and $d(v,y)\le d(v,z)+d(z,y)$.  Hence,
$d(x,u)+d(v,y)\le d(x,z)+d(z,u)+d(v,z)+d(z,y)=d(x,y)+d(u,v)$.
Likewise, $d(x,v)+d(y,u)\le d(x,y)+d(u,v)$.
\end{proof}

The following lemma is a very particular case of a result of~\cite{Iv}
established for all $n$-dimensional uniquely geodesic spaces: 
\begin{lemma} \label{Helly} (Helly property)
 Any collection ${\mathcal C}=\{ C_i: i\in I\}$ of compact convex sets
of ${\mathcal S}$ has a nonempty intersection provided any three sets
of ${\mathcal C}$ have a nonempty intersection. In particular, any
collection of closed balls $\mathcal B$ of ${\mathcal S}$ has a
nonempty intersection provided any three balls of ${\mathcal B}$
intersect.
\end{lemma}

For a compact set $S$ and a point $u\in S$, the {\it eccentricity} of
$u$ in $S$ is $e_S(u)=\max \{ d(u,v): v\in S\}.$
The {\it diameter} diam$(S)$ of $S$ is the maximum eccentricity of a
point $u$ of $S$, i.e., diam$(S)=\max\{ d(u,v): u,v\in S\}$.

\begin{lemma} \label{diameter}
For any compact set $S$ of $\mathcal S$, any point $u\in S$ has the
same eccentricity in the sets conv$(S)$ and $S$. Moreover, the sets
$S$ and conv$(S)$ have the same diameter.
\end{lemma}

\begin{proof}
Let $r:=e_S(u)$ and $R:=\mbox{diam}(S)$.  The set conv$(S)$ can be
constructed as the directed union of the sets $S_0=S\subseteq
S_1\subseteq S_2\subseteq \ldots,$ where $S_{i}=\bigcup_{x,y\in
S_{i-1}} [x,y]$.  By induction on $i$ we will prove that
$e_{S_i}(u)=r$ and diam$(S_{i})=R$. This is obvious for $i=0$. Suppose now $i>0$.
Suppose this holds for all $j<i$ and pick any two points $x,y\in S_i$. By
the definition of $S_i$, there exist four (not necessarily distinct) points $x',x'',y',y''\in S_{i-1}$ such
that $x\in [x',x'']$ and $y\in [y',y'']$. Since diam$\{
x',x'',y',y''\}\le$ diam$(S_{i-1})=R$, we deduce that $x',x''\in B_{R}(y')\cap B_{R}(y'')$. By
convexity of balls, $x\in B_{R}(y')\cap B_{R}(y'')$, i.e.,
$d(x,y'),d(x,y'')\le R$. Hence $y',y''\in B_{R}(x)$. Consequently,
since $y\in [y',y'']$ and $B_R(x)$ is convex, $d(x,y)\le R$, i.e.,
diam$(S_i)=R$. Analogously, since $d(u,x'),d(u,x'')\le r$, the
convexity of the ball $B_{r}(u)$ implies that $d(u,x)\le r$, whence
$e_{S_i}(u)=r$.
\end{proof}

For a point $u$ and a geodesic segment $[x,y]$, the {\it shade} of
$[x,y]$ with respect to $u$ is the
set
\begin{displaymath}
\Sh_u(x,y):= \{ p \in {\mathcal S}:
     [u,p] \cap [x,y] \neq \varnothing
     \mbox{ and some line } (u,p) \mbox{ separates } x \mbox{ and } y\}.
\end{displaymath}
The second condition in the definition of $\Sh_u(x,y)$, about a line
separating $x$ from $y$ might seem irrelevant, but in a Busemann
surface, two lines may be tangent without crossing each other
(as in the conditions of Lemma \ref{two-lines}). In particular, if
$(u,p)$ is tangent to $[x,y]$, then $x$ and $y$ are not necessarily
separated by $(u,p)$.

The {\it shade} $\St_u(x,y,z)$ of a triangle $\Delta(x,y,z)$
with respect to a point $u \notin \Delta(x,y,z)$ is the union of the
shades of its three sides with respect to $u$:
\begin{displaymath}
\Sh_u(x,y,z):= \Sh_u(x,y)\cup \Sh_u(y,z)\cup \Sh_u(z,x).
\end{displaymath}

\begin{lemma}\label{lemma:triangle-shade}
  Every point  $p \in \St_u(x,y,z) \setminus \Delta(x,y,z)$ is contained
  in two of the three shades $\Sh_u(x,y)$, $\Sh_u(y,z)$, and
  $\Sh_u(x,z)$.
\end{lemma}

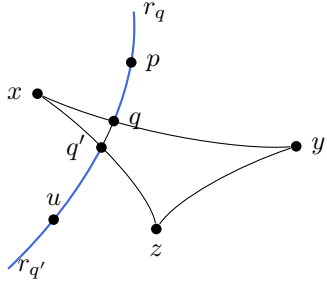
\begin{figure}
  \begin{center} \small
    \begin{tikzpicture}[x=0.6cm,y=0.6cm]
      \draw
        (0.300,0.733) .. controls (2.133,2.483) and (3.233,4.617) .. (3.083,6.433)
        node[pos=0,right] {$r_{q'}$}
        node[pos=1,right] {$r_q$}
        ;

      \draw[thick,draw=RoyalBlue]
        (0.300,0.733) .. controls (1.175,1.568) and (1.883,2.491) .. (2.367,3.424);
      \draw[thick,draw=RoyalBlue]
        (2.641,4.009) .. controls (2.986,4.839) and (3.147,5.664) .. (3.083,6.433);

      \node[stpoint,label=left:$x$] (x) at (0.950,4.617) {};
      \node[stpoint,label=right:$y$] (y) at (6.683,3.450) {};
      \node[stpoint,label=below:$z$] (z) at (3.583,1.617) {};
      \node[stpoint,label=above:$u$] (u) at (1.308,1.824) {};
      \node[stpoint,label=right:$p$] (p) at (3.027,5.311) {};
      \node[stpoint,label=right:$q$] (q) at (2.641,4.009) {};
      \node[stpoint,label=left:$q'$] (q') at (2.367,3.424) {};
      \draw (x) .. controls (2.783,3.783) and (5.383,3.350) .. (y);
      \draw (x) .. controls (2.317,3.667) and (3.467,2.333) .. (z);
      \draw (y) .. controls (5.383,3.017) and (4.033,2.283) .. (z);
    \end{tikzpicture}
  \end{center}
  \caption{Illustration for Lemma~\ref{lemma:triangle-shade}.}
  \label{fig:triangle-shade}
\end{figure}

\begin{proof}
  Let $\ell$ be a geodesic extension of $[u,p]$; since 
  $p \in \St_u(x,y,z)$ we may assume that $\ell$ separates $x$ and
  $y$.  By homeomorphism to $\mathbb{R}^2$, $\ell$ must also separate
  $x$ from $z$ or $z$ from $y$, say the first. Thus both $[x,y]$ and
  $[x,z]$ are intersected by $\ell$. Choose $q \in \ell \cap [x,y]$
  and $q' \in \ell \cap [x,z]$. Then $[q,q'] \subseteq \Delta(x,y,z)$, 
  by convexity of triangles.

  Since $u, p \notin \Delta(x,y,z)$, $u$ and $p$ are each contained
  in one of the rays $r_{q'}$ and $r_q$, where $r_q, r_{q'}$ are
  defined in such a way that $r_{q'} \cup [q',q] \cup r_q = \ell$ and the rays 
  $r_q$ and $r_{q'}$ are disjoint; see Figure \ref{fig:triangle-shade}. 
  If both $u$ and $p$ are contained in the 
  same ray, say $r_{q'}$, then as
  $[u,p] \cap \Delta(x,y,z) \neq \varnothing$ and
  $[q,q'] \subseteq \Delta(x,y,z)$, one of $u$ and $p$ would be in
  $\Delta(x,y,z)$ by convexity of triangles, and this would be a
  contradiction.  Hence $u$ and $p$ are in distinct rays. This implies
  that $q' \in [u,p]$, $[u,p]$ intersects $[x,z]$, whence
  $p \in \St_u(x,z)$.
\end{proof}

\begin{lemma} \label{shade} For any point $u$, any geodesic segment
  $[x,y]$ not containing $u$, and any triangle $\Delta(x,y,z)$ not
  containing $u$, the shades $\Sh_u(x,y)$ and $\St_u(x,y,z)$ are
  convex.
\end{lemma}

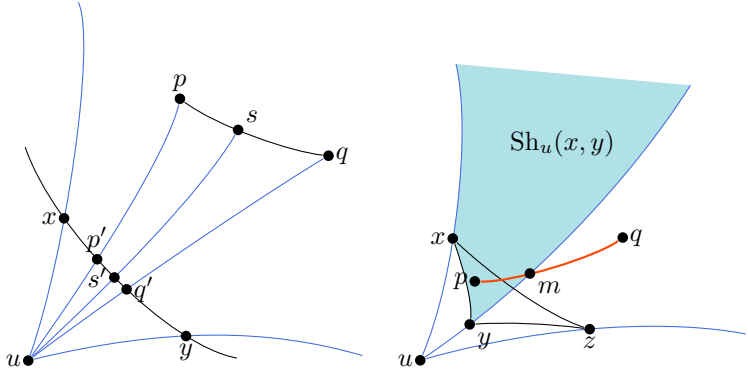
\begin{figure}
  \begin{center}
    \small
    \begin{tabular}{cc}
      \begin{tikzpicture}[x=0.5cm,y=0.5cm,label distance=-3pt]
        \draw (0.167,5.800) .. controls (0.933,3.550) and (4.017,0.500) .. (5.800,0.183);
        \node[stpoint,label=left:$u$] (u) at (0.250,0.117) {};
        \draw[draw=RoyalBlue] (u) .. controls (3.967,1.017) and (6.667,0.967) .. (9.133,0.250);
        \draw[draw=RoyalBlue] (u) .. controls (1.250,2.967) and (2.067,8.583) .. (1.583,9.650);
        \node[stpoint,label=right:$q$] (q) at (8.233,5.567) {};
        \node[stpoint,label=$p$] (p) at (4.283,7.083) {};
        \draw[draw=RoyalBlue] (u) .. controls (3.300,2.417) and (7.233,4.983) .. (q);
        \draw[draw=RoyalBlue] (u) .. controls (2.100,2.683) and (4.000,5.767) .. (p);
        \draw (p) .. controls (5.300,6.333) and (7.250,5.667) .. (q);
        \node[stpoint,label=above right:$s$] (s) at (5.830,6.254) {};
        \draw[draw=RoyalBlue] (u) .. controls (2.517,2.217) and (4.850,4.600) .. (s);
        \node[stpoint,label=left:$s'$] (s') at (2.540,2.328) {};
        \node[stpoint,label=right:$q'$] (q') at (2.865,2.012) {};
        \node[stpoint,label=above:$p'$] (p') at (2.082,2.813) {};
        \node[stpoint,label=below:$y$] (y) at (4.445,0.769) {};
        \node[stpoint,label=left:$x$] (x) at (1.202,3.902) {};
      \end{tikzpicture}
      &
      \begin{tikzpicture}[x=0.5cm,y=0.5cm,label distance=-3pt]
        \fill[fill=PowderBlue]
          (7.400,7.533) .. controls (6.040,5.394) and (3.958,3.062) .. (1.531,1.172)
          .. controls (1.617,1.733) and (1.350,2.717) ..
          (1.081,3.457) .. controls (1.367,5.234) and (1.417,7.060) .. (1.167,8.100);
        \node[stpoint,label=left:$u$] (u) at (0.217,0.217) {};
        \draw[draw=RoyalBlue] (u) .. controls (3.767,1.217) and (7.000,1.283) .. (8.867,0.917);
        \draw[draw=RoyalBlue] (u) .. controls (3.200,2.233) and (5.800,5.017) .. (7.400,7.533);
        \draw[draw=RoyalBlue] (u) .. controls (1.167,2.450) and (1.583,6.367) .. (1.167,8.100);
        \node[stpoint,label=right:$q$] (q) at (5.600,3.483) {};
        \node[stpoint,label=left:$p$] (p) at (1.667,2.317) {};
        \node[stpoint,label=below right:$y$] (y) at (1.531,1.172) {};
        \node[stpoint,label=below:$z$] (z) at (4.721,1.045) {};
        \node[stpoint,label=left:$x$] (x) at (1.081,3.457) {};
        \draw (y) .. controls (2.500,1.233) and (3.583,1.200) .. (z);
        \draw (x) .. controls (2.000,2.583) and (3.517,1.433) .. (z);
        \draw (x) .. controls (1.350,2.717) and (1.617,1.733) .. (y);
        \draw[draw=OrangeRed,thick] (p) .. controls (2.700,2.283) and (4.917,2.983) .. (q);
        \node[stpoint,label=below right:$m$] (m) at (3.132,2.526) {};
        \draw (4,6) node {$\Sh_u(x,y)$};
      \end{tikzpicture}
    \end{tabular}
  \end{center}
  \caption{Illustrations for the proof of Lemma~\ref{shade}.}
  \label{fig:shade-proof}
\end{figure}

\begin{proof}
  Let $p,q \in \Sh_u(x,y)$ and $s \in [p,q]$ (see
  Figure~\ref{fig:shade-proof}, left). We may assume
  $s \notin \{p,q\}$. Let $p' \in [x,y] \cap [u,p]$ and
  $q' \in [x,y] \cap [u,q]$. Suppose without loss of generality that
  $x,p',q',y$ occur in this order on $[x,y]$. By Pasch axiom there
  exists a point $s' \in [p',q'] \cap [u,s]$. Let $\ell$ be some line
  extending $[u,s]$.

  If $\ell$ is tangent to $(p,q)$ at $s$, by Lemma~\ref{two-lines},
  $s \in [u,p]$ or $s \in [u,q]$, and then there is a line $(x,p)$ or
  $(x,q)$ separating $y$ and $z$, and this line extends
  $[x,s]$. Otherwise, $\ell$ separates $p$ and $q$. But $\ell$ does
  not separate $x$ and $p$ (witnessed by the curve with support
  $[x,p'] \cup [p',p]$), and similarly does not separates $y$ and
  $q$. By homeomorphism to $\mathbb{R}^2$, $\ell$ separates
  $\mathcal{S}$ into exactly two connected components, hence $\ell$ separates
  $x$ and $y$. Thus $s \in \Sh_u(x,y)$, establishing the convexity of
  $\Sh_u(x,y)$.

  Now we will prove the convexity of $\St_u(x,y,z)$. If each of $p$ and $q$ is
  not contained in $\Delta(x,y,z)$, then by
  Lemma~\ref{lemma:triangle-shade}  both $p$ and $q$ belong to a common
  shade of the sides of $\Delta$. Since this shade is convex,
  $[p,q] \subset \St_u(x,y,z)$. If both $p$ and $q$ are in
  $\Delta(x,y,z)$, as $\Delta(x,y,z) \subset \St_u(x,y,z)$, the result
  follows by convexity of the triangle $\Delta(x,y,z)$ (Lemma \ref{triangle}).

  Finally, assume that $p \in \Delta(x,y,z)$ and
  $q \notin \Delta(x,y,z)$ (see Figure~\ref{fig:shade-proof},
  right). Let $p$ belong to the shade of $[x,y]$. By Lemma
  \ref{lemma:triangle-shade}, $q$ is in the shades of at least two
  sides.  If one of these sides is $[x,y]$, then we are done. So,
  suppose that $q\notin \Sh_u(x,y)$ and
  $q\in \Sh_u(y,z)\cap \Sh_u(z,x)$.  Let
  $[p,m]: = [p,q] \cap \Sh_u(x,y)$.  If $m \in [x,y]$, let $m'$ be a
  point of $\partial\Delta(x,y,z)$ such that $[m,m']$ is the
  intersection of $\Delta(x,y,z)$ with some line extending $[p,q]$. In
  particular, $m'$ is on a side distinct from $[x,y]$, say
  $m'\in [x,z]$.  Since $p \in [q,m']$ and $m',q\in \Sh_u(x,z)$,
  $p \in \Sh_u(x,z)$ by convexity of $\Sh_u(x,z)$, and we are done.

  Now suppose that the point $m$ is on the boundary of $\Sh_u(x,y)$
  and not on $[x,y]$. Since $m\in \Sh_u(x,y)$, there exists a line
  $\ell$ extending $[u,m]$ and separating $x$ from $y$. By the
  definition of Busemann spaces, if $\ell$ does not pass via $x$ or
  $y$, then for any point $m'$ in a small enough neighborhood of $m$,
  the geodesic $[u,m']$ also intersects $[x,y]$ and a line extending
  $[u,m']$ will separate $x$ from $y$, whence $m'\in \Sh_u(x,y)$. But
  this contradicts the choice of $m$ as the point such that
  $[p,m] = [p,q] \cap \Sh_u(x,y)$.  Indeed, the extension of $[p,m]$
  through $m$ in the direction of $q$ will contain points $m'$ of
  $\Sh_u(x,y)$.  Hence the line $\ell$ passes via $x$ or $y$, i.e.,
  $m$ is in $\Sp_u(x)$ or $\Sp_u(y)$. Hence $m$ belongs to the shade
  of $[x,z]$ or $[y,z]$.  Since $q\in \Sh_u(y,z)\cap \Sh_u(z,x)$, by
  convexity of that shade we conclude that
  $[m,q] \subset \St_u(x,y,z)$. Since by construction of $m$,
  $[p,m] \subset \Sh_u(x,y) \subset \St_u(x,y,z)$, we obtain that
  $[p,q] \subset \St_u(x,y,z)$, establishing the convexity of
  $\St_u(x,y,z)$.
\end{proof}

\begin{lemma} \label{separa}
  If $v \notin \Delta(x,y,z)$, then there exists a line $\ell$
  extending a side of $\Delta(x,y,z)$ and separating  $v$ and
  $\Delta(x,y,z)$.
\end{lemma}

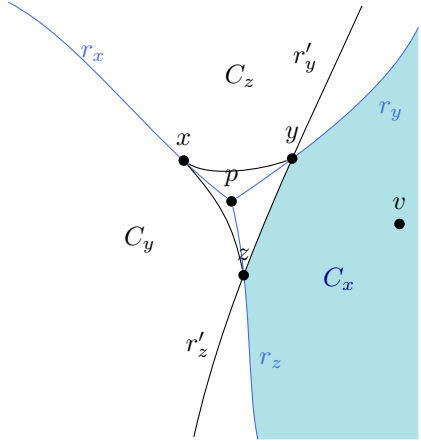
\begin{figure}
  \begin{center} \small
    \begin{tikzpicture}[x=1cm,y=1cm]
      \fill[PowderBlue]
        (7.150,7.050) .. controls (6.893,6.510) and (6.309,5.930) ..
        (5.473,5.303) .. controls (5.251,4.803) and (5.029,4.282) ..
        (4.828,3.751) .. controls (4.922,2.876) and (4.907,2.066) ..
        (5.017,1.567) -- (7.150,1.567);

      \draw
        (6.400,7.333) .. controls (5.683,5.733) and (4.617,3.650) .. (4.167,1.600)
        node[pos=0.2,above left] {$r'_y$}
        node[pos=0.8,left] {$r'_z$};

      \node[stpoint,label=$v$] (v) at (6.900,4.433) {};
      \node[stpoint,label=$p$] (p) at (4.667,4.733) {};
      \draw[RoyalBlue]
         (p) .. controls (4.950,3.517) and (4.867,2.250) .. (5.017,1.567)
         node[pos=0.6,right,text=RoyalBlue] {$r_z$};
      \draw[RoyalBlue]
        (p) .. controls (5.933,5.583) and (6.817,6.350) .. (7.150,7.050)
        node[pos=0.6,below right,text=RoyalBlue] {$r_y$};
      \draw[RoyalBlue]
        (p) .. controls (3.583,5.550) and (2.633,6.900) .. (1.700,7.383)
        node[pos=0.6,above,text=RoyalBlue] {$r_x$};

      \node[stpoint,label=$x$] (x) at (4.032,5.276) {};
      \node[stpoint,label=$z$] (z) at (4.828,3.751) {};
      \node[stpoint,label=$y$] (y) at (5.473,5.303) {};
      \draw (x) .. controls (4.517,4.717) and (4.683,4.433) .. (z);
      \draw (x) .. controls (4.383,5.050) and (5.067,5.150) .. (y);

      \draw (4.783,6.400) node {$C_z$};
      \draw (3.450,4.217) node {$C_y$};
      \draw[NavyBlue] (6.100,3.700) node {$C_x$};
    \end{tikzpicture}
  \end{center}
  \caption{Illustration for the proof of Lemma~\ref{separa}.}
  \label{fig:separa}
\end{figure}

\begin{proof}
  We may assume $x$, $y$ and $z$ are not aligned, otherwise any line $\ell$
  containing them would separate the triangle from any point.

  Let $p \in \Delta(x,y,z) \setminus \{x,y,z\}$. Let $r_x$ be a ray
  emanating from $x$ not going through $p$ on a line $(p,x)$. Define
  similarly $r_y$ and $r_z$. Those three rays are distinct because
  $x$, $y$ and $z$ are not aligned. Then by homeomorphism to
  $\mathbb{R}^2$, $r_x \cup r_y \cup r_z \cup \partial\Delta(x,y,z)$
  separates the surface $\mathcal{S}$ into 4 connected components, one
  of them being $\Delta(x,y,z)$ (see Figure~\ref{fig:separa}). We may
  assume that $v$ is in the closure $C_x$ of the component with
  boundary $\gamma := r_y \cup [y,z] \cup r_z$. Hence $\gamma$
  separates $v$ from $\Delta(x,y,z)$.

  Let $(y,z)$ be an extension of $[y,z]$, let $r'_y$ be the ray of
  $(y,z)$ from $y$ not containing $z$, and $r'_z$ be the ray of
  $(y,z)$ from $z$ not containing $y$, so that
  $(y,z) = r'_y \cup [y,z] \cup r'_z$. Then we may assume that $r'_y$
  does not intersect the interior of $C_x$. Indeed, otherwise $(y,z)$
  is tangent to $(p,y)$ on $y$, hence by Lemma~\ref{two-lines} we
  could choose $(y,z)$ such that $r_y = r'_y$. Similarly we may assume
  $r'_z$ does not intersect the interior of $C_x$. Hence the line
  $(y,z)$ separates $C_x$ from $\Delta(x,y,z)$.
\end{proof}

\subsection{Proof of Proposition~\ref{prop:three-balls}} \label{subsection:three-balls}

In this subsection we will prove the following Proposition \ref{prop:three-balls}:

\medskip\noindent
{\bf Proposition 2}. {\it Let $S$ be a compact subset of a Busemann surface $({\mathcal S},d)$
and suppose that the diameter of $S$ is at most $2{\delta}$. Then $S$
can be covered with 3 balls of radius ${\delta}$, i.e., $\rho(S)\le
3$.
}

\begin{proof} Let $S$ be a compact subset of $(\mathcal S,d)$ and suppose that the
diameter of $S$ is at most $2{\delta}$. Since by Lemma~\ref{diameter},
the diameter of conv$(S)$ coincides with the diameter of $S$ and
conv$(S)$ is compact, we will further assume without loss of
generality that $S$ is convex.  We will prove that $S$ can be covered with
three balls of radius ${\delta}$. Since diam$(S)\le
2{\delta}$, any two balls centered at points of $S$ intersect. If any
three such balls intersect, then Lemma~\ref{Helly} implies that
$\bigcap_{x\in S}B_{\delta}(x)\ne \varnothing$ and if $v$ is an
arbitrary point from this intersection, then $S\subseteq
B_{\delta}(v)$.  Therefore, further we can suppose that $S$ contains
triplets of points such that the ${\delta}$-balls centered at these
points have an empty intersection. We will call such triplets {\it
critical}.

Let $x,y,z\in S$ be an arbitrary triplet of points of $S$. Denote by
$x^*,y^*,$ and $z^*$ the midpoints of the geodesics $[y,z],[x,z],$ and
$[x,y]$, respectively. Since $d(x,y),d(y,z),d(z,x)\le 2{\delta}$, from
Proposition A we conclude that
$d(x^*,y^*),d(y^*,z^*),d(z^*,x^*)\le {\delta}$.  Let
$A_x:=\Delta(x,y,z)\cap B_\delta(y) \cap B_\delta(z),
A_y:=\Delta(x,y,z)\cap B_\delta(x) \cap B_\delta(z),$ and
$A_z:=\Delta(x,y,z)\cap B_\delta(x) \cap B_\delta(y)$.  These sets
are compact (as the intersection of compact sets) and nonempty
(because $x^*\in A_x, y^*\in A_y$, and $z^*\in A_z$). Among all
triplets of points, one from each of the sets $A_x,A_y,$ and $A_z$,
let $x',y',z'$ be a triplet with the minimum perimeter $\pi(x',y',z')$
of $\Delta(x',y',z')$.  Such a triplet exists because the sets
$A_x,A_y$, and $A_z$ are compact.  If the triplet $x,y,z$ is not
critical, then the points $x',y',z'$ coincide.  We will call
$\Delta(x',y',z')$ a {\it critical triangle} for the triplet
$x,y,z$.

The \emph{roadmap of the proof} is as follows: we prove that the three
$\delta$-balls centered at $x'$, $y'$, and $z'$ cover the whole set
$S$ (Claim 6). We proceed by contradiction and assume that there is an
uncovered point $v \in S$. The proof depends on the position of
$v$. The first part of the proof is to exhibit a suitable partition of
the set $S$. First, the triangle $\Delta(x,y,z)$ is subdivided into
seven smaller triangles (Claim 5, see Figure~\ref{fig:prop2-cases}
Case 1), and we show that each of them is covered. Thus $v$ must be
outside $\Delta(x,y,z)$. If one of the segments $[x,v]$, $[y,v]$, and
$[z,v]$ intersects the critical triangle $\Delta(x',y',z')$, then
again $v$ is covered (Figure~\ref{fig:prop2-cases} Cases 2 and
3). Finally, in the remaining cases (Figure~\ref{fig:prop2-cases} Case
4), we show that $v$ with two points among $x, y, z$ define a critical
triangle with a larger perimeter, contradicting the choice of
$\Delta(x,y,z)$. Claims 1--5 are about the geometry of $S$ with
respect to the defined points. Claim 6 examines the four possible
locations of $v$, illustrated in Figure~\ref{fig:prop2-cases}, and
discards each of them.

We continue with simple properties of critical triplets and
their critical triangles:

\begin{claim} \label{claim:critical-triangle}
If $x,y,z$ is a critical triplet of
$S$, then (a) the triangle $\Delta(x',y',z')$ is non-degenerated and
(b) $x'\in \partial B_{\delta}(y)\cap \partial B_{\delta}(z), y'\in
\partial B_{\delta}(z)\cap \partial B_{\delta}(x)$, and $z'\in
\partial B_{\delta}(x)\cap \partial B_{\delta}(y)$.
\end{claim}

\begin{proof} The assertion (a) follows from the convexity of balls: if
$\Delta(x',y',z')$ is degenerated and say $y'\in [x',z']$, since
$x',z'\in B_{\delta}(y)$, from the convexity of $B_{\delta}(y)$ we
conclude that $y'\in B_{\delta}(y)$, contrary to the assumption that
$x,y,z$ is critical.

To prove (b), suppose by way of contradiction that $y'\notin \partial
B_{\delta}(x)$, i.e., $d(x,y')<\delta$. Then there exists an
$\varepsilon>0$ such that $B^{\circ}_{\varepsilon}(y')\subset
B_{\delta}(x)$. On the other hand, the intersection
$B_{\varepsilon}^{\circ}(y')\cap \Delta(x',y',z')$ is different from $y'$.
Since $y',x'\in B_{\delta}(z)$, the
convexity of $B_{\delta}(z)$ implies that $[y',x']\subset
B_{\delta}(z)$. Therefore, we can find a point $y''\in [y',x']\cap
B_{\varepsilon}^{\circ}(y')$ different from $y'$. Then $y''\in
\Delta(x',y',z')\subseteq \Delta(x,y,z)$ and $y''$ still belongs
to the intersection $B_{\delta}(x)\cap B_{\delta}(z)$. Since
$\Delta(x',y',z')$ is non-degenerated, by Lemma~\ref{perimeter}, we
obtain $\pi(x',y'',z')<\pi(x',y',z'),$ contrary to the choice of the
points $x',y',z'$. This finishes the proof of Claim~\ref{claim:critical-triangle}.
\end{proof}

\medskip
Now, among all triplets of $S$ select a triplet $x,y,z$ for which the
perimeter of the critical triangle $\Delta(x',y',z')$ is as large as
possible. Notice that such a triplet necessarily exists since the
perimeter function $\pi: S\times S\times S\rightarrow {\mathbb R}^+$
is continuous because $S$ is convex and attain a maximum because $S$
is compact. Clearly, $x,y,z$ is a critical triplet of $S$.

\begin{figure}\label{figure1}
\begin{tikzpicture}[x=0.8cm,y=0.8cm]
\def \radius{4.2}
\def \radiusBall{\radius*0.92}
\node (c) at (0,0) {};
\draw (c) node [stpoint] (y) at +(-30:\radius) {};
\draw (c) node [stpoint] (x) at +(-30-120:\radius) {};
\draw (c) node [stpoint] (z) at +(-30-120-120:\radius) {};
\node [stpoint] (y*) at ($(x)!0.5!(z)$) {};
\node [stpoint] (x*) at ($(y)!0.5!(z)$) {};
\node [stpoint] (z*) at ($(x)!0.5!(y)$) {};
\draw (x) -- (y);
\draw (y) -- (z);
\draw (x) -- (z);
\draw (x*) -- (y*);
\draw (y*) -- (z*);
\draw (x*) -- (z*); 
\draw[name path = ball_x,densely dotted] (x) +(-30:\radiusBall) arc (-30:90:\radiusBall);
\draw[name path = ball_y,densely dotted] (y) +(90:\radiusBall) arc (90:210:\radiusBall);
\draw[name path = ball_z,densely dotted] (z) +(210:\radiusBall) arc (210:330:\radiusBall);
\draw[name intersections = {of = ball_x and ball_y}] (intersection-1) node (z') [stpoint] {};
\draw[name intersections = {of = ball_y and ball_z}] (intersection-1) node (x') [stpoint] {};
\draw[name intersections = {of = ball_x and ball_z}] (intersection-1) node (y') [stpoint] {};
\draw (x') -- (y');
\draw (y') -- (z');
\draw (x') -- (z');
\path (x') node[above] {$x'$}; 
\path (y') node[above] {$y'$};
\path (z') node[right] {$z'$};
\path (x) node[below left] {$x$};
\path (y) node[below right] {$y$};
\path (z) node[above] {$z$};
\path (y*) node[above left] {$y^*$};
\path (x*) node[above right] {$x^*$};
\path (z*) node[below] {$z^*$};
\end{tikzpicture}
 \caption{The choice of points $x',y',z'$ in Proposition~\ref{prop:three-balls}.}
\end{figure}
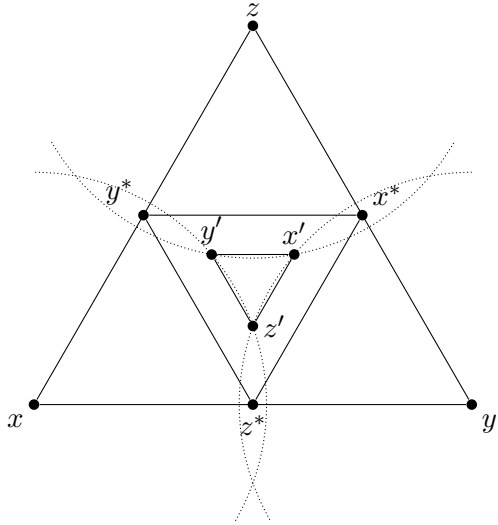

\begin{claim} \label{claim:containment-triangle} $\Delta(x',y',z')\subseteq \Delta(x^*,y^*,z^*)$. In particular,
$d(x',y'),d(y',z'),d(z',x')\le \delta$.
\end{claim}

\begin{proof}
Since $\Delta(x^*,y^*,z^*)$ is convex, it suffices to show that
$x',y',z'\in \Delta(x^*,y^*,z^*)$. By their definition, the points
$x',y',z'$ belong to $\Delta(x,y,z)$. The triangle $\Delta(x,y,z)$
is the union of four triangles
$\Delta(x,y^*,z^*),\Delta(x^*,y,z^*),\Delta(x^*,y^*,z),$ and
$\Delta(x^*,y^*,z^*)$.  Suppose by way of contradiction that one of
the points $x',y',z'$ is located in $\Delta(x,y^*,z^*)\setminus
[y^*,z^*]$. Since $d(x,y^*),d(x,z^*)\le \delta$, by the convexity of
$B_{\delta}(x)$, $d(x,v)\le \delta$ for any point $v\in
[y^*,z^*]$. Now, if a point $w$ belongs to
$\Delta(x,y^*,z^*)\setminus [y^*,z^*]$, then extending the geodesic
$[x,w]$ through $w$ we will find a point $w'\in [y^*,z^*]$ such that
$w\in [x,w']$. Since $d(x,w')\le \delta$, we conclude that
$d(x,w)<\delta$. Consequently, neither of the points $x',y',z'$ can
belong to $\Delta(x,y^*,z^*)\setminus [y^*,z^*]$ (because each of
them belongs to two spheres and does not belong to the third
ball). Analogously, one can prove that $x',y',z'$ do not belong to
$\Delta(x^*,y,z^*)\setminus [x^*,z^*]$ and to
$\Delta(x^*,y^*,z)\setminus [x^*,y^*]$. Consequently, $x',y',z'\in
\Delta(x^*,y^*,z^*).$ The second assertion follows from Lemma
~\ref{perimeter2}.  This establishes Claim~\ref{claim:containment-triangle}.
\end{proof}

\medskip
We continue with a monotonicity property of the shade $\St_x(y',z')$.
Let $s(y') \in [y,z] \cap
[x,y')$ and $s(z')\in [y,z]\cap [x,z')$, where $[x,y')$ and $[x,z')$
are two rays with origin $x$ passing through $y'$ and $z'$,
respectively.  We will call $s(y')$ and $s(z')$ the {\it shadows} of
$y'$ and $z'$ in $[y,z]$ (or in any line $(y,z)$ extending $[y,z]$).
Analogously, one can define the shadow $s(p)$ in $[y,z]$
of any point $p\in [y',z']$ or of any point $p\in \Delta(x,y,z)$.

\begin{claim}\label{claim:order-shadows}
For any choice of the shadows $s(y')$ and $s(z')$ of $y'$ and $z'$ in
$[y,z],$ the points $y,s(z'),s(y'),z$ occur in this order on $[y,z]$.
\end{claim}

\begin{proof} Suppose by way of contradiction that $y, s(y'), s(z'), z$ occur in
this order on $[y,z]$. Then $y' \in [x,s(y')] \subset
\Delta(x,y,s(z'))$. If $y' \in \Delta(x,y,z')$, then by
Lemma~\ref{perimeter} (perimeters of triangles with basis $[x,y]$), we
have
$$2 \delta < d(x,y') + d(y',y) \leq d(x,z') + d(z',y) = 2 \delta,$$ a
contradiction. On the other hand, if $y' \in \Delta(z',y,s(z'))$,
then $[y',z]$ intersects $[x,s(z')]$ and $[z',s(z')]$. Consequently,
$z' \in \Delta(x,y',z)$ and this case is symmetric to the first
case. Since $\Delta(x,y,z')$ and $\Delta(z',y,s(z'))$
cover $\Delta(x,y,s(z'))$, this finishes the proof of
Claim~\ref{claim:order-shadows}.
\end{proof}

\begin{claim}\label{claim:shadow-convexity}
  If $p,q \in \Delta(x,y,z), v \in \St_x(p,q),$ and
  $v' \in [y,z] \cap (x,v)$, where $(x,v)$ is a line passing via $x$
  and $v$ and separating $p$ and $q$, then there exist shadows $s(p)$
  and $s(q)$ of $p$ and $q$ in $[y,z]$ such that $v' \in
  [s(p),s(q)]$.
\end{claim}

\begin{proof} Pick any shadows $s(p)$ and $s(q)$ of $p$ and $q$ in $[y,z]$. Suppose without loss of
generality that the points $y,s(p),s(q),z$ occur in
this order on $[y,z]$.  Assume that $v'\notin [s(p),s(q)]$, otherwise we are done.
Suppose without loss of generality that
$v'\in [y,s(p)]$.  Since $x,s(p)$, and $s(q)$ all belong to a common closed halfplane
defined by $(x,v')=(x,v)$, the whole triangle $\Delta(x,s(p),s(q))$ also belong to
this halfplane. Since $p,q\in \Delta(x,s(p),s(q))$ and the line
$(x,v)$ separates $p$ and $q$, we conclude that $p\in (x,v)$. This implies
that $p\in [x,v']$ and consequently, $v'$ is a shadow of $p$ in $[y,z]$. Thus
selecting $v'$ as a shadow $s(p)$ of $p$ we are done.
\end{proof}

\begin{claim} \label{claim:seven-triangles}
The seven triangles $$\Delta(x,y,z'), \Delta(x,y',z),
\Delta(x',y,z), \Delta(x,y',z'), \Delta(x',y,z'),
\Delta(x',y',z),\Delta(x',y',z')$$ partition the triangle
$\Delta(x,y,z)$.
\end{claim}

\begin{proof}
  First we show that
  $\Delta(y,z,x') = \Delta(y,z,s_y(x')) \cap \Delta(y,z,s_z(x'))$,
  where $s_y(x')$ and $s_z(x')$ are shadows of $x'$ in $[x,z]$ with
  respect to $y$ and in $[x,y]$ with respect to $z$.  Indeed, since
  $x'\in [y, s_y(x')]\cap [z, s_z(x')]$, by convexity of triangles we
  have
  $\Delta(y,z,x') \subseteq\Delta(y,z,s_y(x')) \cap \Delta(y,z,s_z(x'))$.
  To prove the converse inclusion, let
  $w \in \Delta(y,z,s_y(x')) \cap \Delta(y,z,s_z(x'))$ and suppose
  that $w\notin \Delta(y,z,x')$. Then
  $w \in \Delta(y,z,s_z(x')) \setminus \Delta(y,z,x')
  = \Delta(y,x',s_z(x')) \setminus [y,x']$. Since
  any shadow of $w$ in $[x,z]$ with respect to $y$ belongs to
  $[x,s_y(x')]\setminus \{ s_y(x')\}$, this contradicts
  $w \in \Delta(y,z,s_y(x'))$. In the same way, we can prove analogous
  statements for $\Delta(x,y,z')$ and $\Delta(x,z,y')$. From this and
  Claim~\ref{claim:order-shadows} we deduce that the triangles
  $\Delta(y,z,x'),\Delta(x,y,z'),$ and $\Delta(x,z,y')$ pairwise
  intersect only in the segments $[x,z'] \cap [x,y']$,
  $[y,x'] \cap [y,z']$ and $[z,x'] \cap [z,y']$.

  Let $P$ be the closure of $\Delta(x,y,z) \setminus
  (\Delta(y,z,x')\cup \Delta(x,y,z')\cup \Delta(x,z,y'))$.  Then
  $P$ is a hexagon with vertices $x, y', z, x', y, z'$ and sides
  $[x,y'],[y',z],[z,x'],[x',y],[y,z'],$ and $[z',x]$. We assert that
  $[x',y'], [x',z'],$ and $[y',z']$ are diagonals of $P$ (i.e., belong
  to $P$). If $[x',y']$ is not included in $P$, then $P$ contains a
  vertex in $\Delta(z,x',y')$ different from $z,x',y'$. Clearly,
  this vertex can only be $z'$. But $\Delta(z,x',y') \subseteq
  B_\delta(z)$ and $d(z,z') > \delta$, a contradiction. The three
  diagonals do not cross each other because they pairwise have a
  common extremity. Hence $[x',y'], [x',z'],[y',z']$ triangulate $P$,
  concluding the proof of the claim.
\end{proof}

% The result of the proposition follows from the following assertion.

\begin{claim}\label{claim:ball-cover}
$S \subseteq B_{\delta}(x') \cup B_{\delta}(y') \cup B_{\delta}(z').$
\end{claim}

\begin{figure}
  \begin{center} \small
  \begin{tabular}{cc}
      \begin{tikzpicture}[x=0.8cm,y=0.8cm]
\def \radius{3.5}
\def \radiusBall{\radius*0.92}
\node (c) at (0,0) {};
\draw (c) node [stpoint] (y) at +(-30:\radius) {};
\draw (c) node [stpoint] (x) at +(-30-120:\radius) {};
\draw (c) node [stpoint] (z) at +(-30-120-120:\radius) {};
\draw (x) -- (y); 
\draw (y) -- (z);
\draw (x) -- (z);
\draw[name path = ball_x,densely dotted] (x) +(-10:\radiusBall) arc (-10:70:\radiusBall);
\draw[name path = ball_y,densely dotted] (y) +(110:\radiusBall) arc (110:190:\radiusBall);
\draw[name path = ball_z,densely dotted] (z) +(230:\radiusBall) arc (230:310:\radiusBall);
\draw[name intersections = {of = ball_x and ball_y}] (intersection-1) node (z') [stpoint] {};
\draw[name intersections = {of = ball_y and ball_z}] (intersection-1) node (x') [stpoint] {};
\draw[name intersections = {of = ball_x and ball_z}] (intersection-1) node (y') [stpoint] {};
\draw (x') -- (y');
\draw (y') -- (z'); 
\draw (x') -- (z');
\draw (z) -- (y');
\draw (z) -- (x');
\draw (x) -- (z');
\draw (x) -- (y');
\draw (y) -- (z');
\draw (y) -- (x');
\path (x') +(-0.3,0.3) node {$x'$};
\path (y') +(0.3,0.3) node {$y'$};
\path (z') node[above=0.2, right] {$z'$};
\path (x) node[below left] {$x$};
\path (y) node[below right] {$y$};
\path (z) node[above] {$z$};
\end{tikzpicture}
     &
      \begin{tikzpicture}[x=0.7cm,y=0.7cm,label distance=-3pt]
        \fill[fill=PowderBlue]
          (6.867,6.533) .. controls (6.519,5.800) and (5.860,4.938) ..
          (5.038,4.069) .. controls (5.561,3.150) and (6.176,2.308) ..
          (6.826,1.605) .. controls (7.892,1.735) and (8.788,1.810) .. (9.333,1.833);
        \node[stpoint,label=right:$z$] (z) at (8.000,0.517) {};
        \node[stpoint,label=above:$y$] (y) at (4.183,5.867) {};
        \node[stpoint,label=left:$x$] (x) at (0.350,0.300) {};
        \node[stpoint,label=right:$x'$] (x') at (5.250,2.350) {};
        \node[stpoint,label=below:$y'$] (y') at (4.072,1.190) {};
        \node[stpoint,label=left:$z'$] (z') at (3.238,2.374) {};
        \node[stpoint,label=right:$v$] (v) at (6.717,4.433) {};
        \node[stsmallpoint,label=right:$v''$] (v'') at (4.085,2.143) {};
        \node[stsmallpoint,label=left:$v'$] (v') at (3.801,1.969) {};

        \draw (x) .. controls (2.983,0.850) and (5.583,0.733) .. (z);
        \draw (y) .. controls (5.033,3.667) and (6.517,1.667) .. (z);
        \draw (x) .. controls (1.800,1.583) and (3.833,4.067) .. (y);
        \draw (z') .. controls (4.017,2.267) and (4.550,2.183) .. (x');
        \draw (y') .. controls (3.933,1.667) and (3.733,1.983) .. (z');
        \draw (x') .. controls (4.417,1.867) and (4.117,1.550) .. (y');

        \draw[draw=RoyalBlue] (x) .. controls (3.650,1.583) and (5.967,3.250) .. (v);
        \draw[draw=RoyalBlue] (x) .. controls (3.400,1.283) and (7.783,1.767) .. (9.333,1.833);
        \draw[draw=RoyalBlue] (x) .. controls (2.867,1.767) and (5.950,4.600) .. (6.867,6.533);

        \draw[thick,draw=OrangeRed]
          (x) .. controls (1.670,0.813) and (2.833,1.388) .. (v')
          node[pos=0.4,below,text=Red] {$\delta$};
        \draw[thick,draw=OrangeRed]
          (v'') .. controls (5.385,2.956) and (6.383,3.800) .. (v)
          node[pos=0.7,above,text=OrangeRed] {$\leq \delta$};
        \draw[thick,draw=OrangeRed]
          (y) .. controls (4.283,4.117) and (4.250,2.883) .. (v'')
          node[pos=0.5,left,text=OrangeRed] {$\delta$};
        \draw[thick,draw=SeaGreen]
          (x') .. controls (5.617,2.950) and (6.217,3.583) .. (v)
          node[pos=0.3,right,text=SeaGreen] {$\leq \delta$};
        \draw[thick,draw=SeaGreen]
          (y) .. controls (4.650,4.017) and (4.983,3.017) .. (x')
          node[pos=0.3,right,text=SeaGreen] {$\delta$};
      \end{tikzpicture}
    \\
      Case 1.
    &
      Case 2.
    \\
    \begin{tikzpicture}[x=0.7cm,y=0.7cm,label distance=-3pt]
      \fill[PowderBlue]
        (8.283,9.683) .. controls (8.063,8.969) and (7.204,7.696) ..
        (6.099,6.305) .. controls (6.396,5.603) and (6.693,4.905) ..
        (6.993,4.237) .. controls (8.149,4.842) and (9.166,5.313) .. (9.817,5.483);
      \node[stpoint,label=left:$x$] (x) at (0.583,0.600) {};
      \draw[draw=RoyalBlue] (x) .. controls (3.167,2.083) and (7.967,5.000) .. (9.817,5.483);
      \draw[draw=RoyalBlue] (x) .. controls (3.367,1.783) and (8.900,2.433) .. (9.667,2.367);
      \draw[draw=RoyalBlue] (x) .. controls (2.950,2.333) and (7.717,7.850) .. (8.283,9.683);
      \node[stpoint,label=right:$s(v')$] (sv') at (6.594,5.144) {};
      \node[stpoint,label=left:$v'$] (v') at (5.011,3.839) {};
      \node[stpoint,label=right:$v$] (v) at (8.367,6.817) {};
      \node[stpoint,label=left:$z'$] (z') at (4.417,2.821) {};
      \node[stpoint,label=below:$y'$] (y') at (5.881,1.952) {};
      \node[stpoint,label=left:$x'$] (x') at (5.568,5.651) {};
      \node[stpoint,label=right:$s(z')$] (sz') at (6.993,4.237) {};
      \node[stpoint,label=right:$s(y')$] (sy') at (7.953,2.245) {};
      \node[stpoint,label=right:$s(x')$] (sx') at (6.099,6.305) {};
      \node[stpoint,label=above:$y$] (y) at (5.150,8.550) {};
      \node[stpoint,label=below:$z$] (z) at (8.617,1.100) {};
      \draw[thick,dotted,draw=OrangeRed]
        (y) .. controls (5.933,7.683) and (7.233,6.950) .. (v);
      \draw[thick,dotted,draw=OrangeRed]
        (x') .. controls (6.550,5.933) and (7.617,6.350) .. (v);
      \draw[thick,draw=OrangeRed]
        (y) .. controls (5.467,7.333) and (5.633,6.150) .. (x')
        node[pos=0.4,right,text=Red] {$= \delta$};
      \draw[thick,draw=OrangeRed]
        (y) .. controls (5.367,5.700) and (5.117,4.450) .. (v')
        node[pos=0.3,left,text=Red] {$\leq \delta$};
      \draw[draw=RoyalBlue] (x) .. controls (3.333,2.350) and (7.733,5.967) .. (v);
      \draw (z') .. controls (4.950,2.600) and (5.533,2.250) .. (y');
      \draw (x') .. controls (5.317,4.433) and (4.867,3.333) .. (z');
      \draw (x') .. controls (5.467,4.167) and (5.533,2.850) .. (y');
      \draw (y) .. controls (6.283,5.917) and (7.433,2.933) .. (z);
      \draw (x) .. controls (2.983,1.283) and (6.867,1.450) .. (z);
      \draw (x) .. controls (2.450,2.400) and (5.000,6.567) .. (y);
      \draw[thick,draw=OrangeRed]
        (v') .. controls (6.648,5.148) and (8.032,6.368) .. (v)
        node[pos=0.7,below,Red] {$\leq \delta$};
    \end{tikzpicture}
    &
      \begin{tikzpicture}[x=0.7cm,y=0.7cm,label distance=-3pt]
        \fill[fill=PowderBlue]
          (2.233,0.883) .. controls (2.871,1.313)  and (3.567,2.245) ..
          (4.204,3.326) .. controls (5.204,3.349) and (6.241,3.345) ..
          (7.168,3.317) .. controls (7.544,2.307) and (7.960,1.449) .. (8.300,0.950);
        \fill[fill=PowderBlue]
          (8.500,7.567) .. controls (8.215,7.227) and (7.697,6.764) ..
          (7.012,6.253) .. controls (7.454,5.622) and (7.932,5.005) ..
          (8.389,4.476) .. controls (9.071,4.566) and (9.582,4.619) .. (9.850,4.633);
        \draw (5.079,4.971) .. controls (5.400,4.617) and (5.600,4.383) .. (5.654,4.054);
        \draw (6.547,5.295) .. controls (6.117,5.083) and (5.417,4.917) .. (5.079,4.971);
        \draw (5.654,4.054) .. controls (5.883,4.483) and (6.200,4.950) .. (6.547,5.295);
        \node[stpoint,label=left:$y$] (y) at (1.267,3.150) {};
        \node[stpoint,label=above:$x$] (x) at (6.150,7.583) {};
        \draw[draw=RoyalBlue] (x) .. controls (6.333,5.150) and (7.517,2.100) .. (8.300,0.950);
        \draw[draw=RoyalBlue] (x) .. controls (5.517,5.600) and (3.717,1.883) .. (2.233,0.883);
        \draw[draw=RoyalBlue] (y) .. controls (5.050,4.067) and (8.917,4.583) .. (9.850,4.633);
        \draw[draw=RoyalBlue] (y) .. controls (4.600,4.350) and (7.633,6.533) .. (8.500,7.567);
        \node[stpoint,label=right:$v$] (v) at (9.350,1.933) {};
        \node[stpoint,label=below:$q$] (q) at (7.168,3.317) {};
        \node[stpoint,label=below:$p$] (p) at (4.204,3.326) {};
        \node[stpoint,label=right:$z$] (z) at (9.783,3.117) {};
        \draw[thick,draw=OrangeRed] (y) .. controls (3.783,3.033) and (7.883,2.350) .. (v);
        \draw[thick,draw=OrangeRed] (x) .. controls (6.950,5.100) and (8.017,3.033) .. (v);
        \node[stpoint,label=below:$v'$] (v') at (8.147,3.275) {};
        \draw (x) .. controls (7.183,5.833) and (8.833,3.783) .. (z);
        \draw (y) .. controls (3.733,3.433) and (8.283,3.383) .. (z);
        \draw[thick,draw=OrangeRed] (x) .. controls (4.950,5.500) and (2.750,3.783) .. (y);
        \draw (5.617,1.967) node {$\Sh_x$};
        \draw (8.650,5.867) node {$\Sh_y$};
        \draw (5.783,4.750) node {$\Delta$};
      \end{tikzpicture}
    \\
      Case 3. & Case 4.
    \\
  \end{tabular}
  \end{center}
  \caption{Illustration of the proof of Claim~\ref{claim:ball-cover}.}
  \label{fig:prop2-cases}
\end{figure}
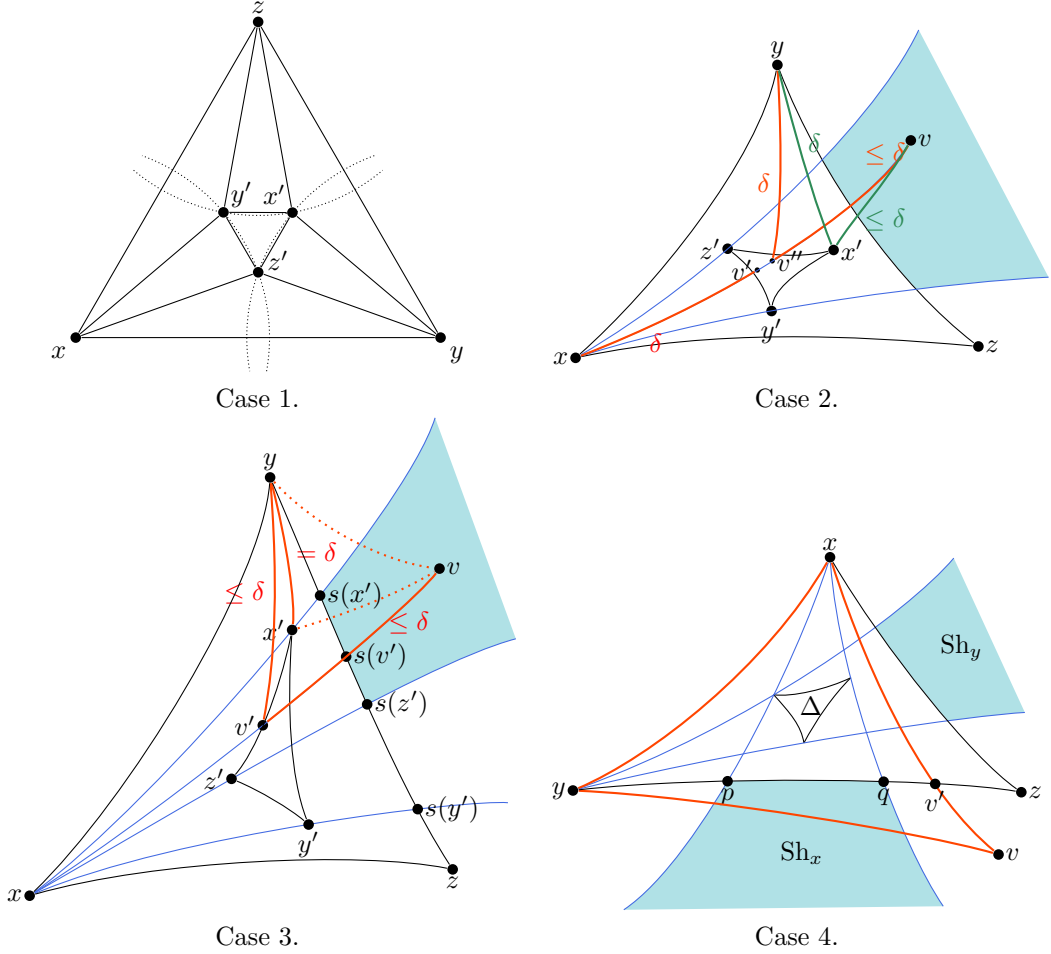

\begin{proof} 
Pick any point $v \in S$. We distinguish four cases,
depending of the location of $v$.

\medskip\noindent {\sf Case 1:} $v\in \Delta(x,y,z)$.

Then $v$ is located in one of the seven triangles defined in
Claim~\ref{claim:seven-triangles}. First suppose that $v\in
\Delta(x',y',z')$. Since by Claim~\ref{claim:containment-triangle}
each side of $\Delta(x',y',z')$ is of length at most ${\delta}$, by
convexity of balls, $\Delta(x',y',z')$ belong to each of the balls
$B_{\delta}(x'),B_{\delta}(y'),$ and $B_{\delta}(z')$, whence
$d(x',v),d(y',v),d(z',v)\le \delta$.

Now suppose that $v\in \Delta(x,y',z')\cup \Delta(x',y,z')\cup
\Delta(x',y',z),$ say $v \in \Delta(x,y',z')$. Analogously to the
previous case, since the sides of the triangle $\Delta(x,y',z')$ are
at most $\delta$, we conclude that $d(y',v),d(z',v) \le
\delta$. Finally, suppose that $v \in \Delta(x,y,z') \cup \Delta(x',y,z)
\cup \Delta(x,y',z),$ say $v \in \Delta(x,y,z')$. Then $x,y \in
B_{\delta}(z')$, whence $v \in \Delta(x,y,z') \subseteq
B_{\delta}(z')$, yielding  $d(z',v) \le \delta$.  This concludes the proof
of Case 1.

\medskip
Further, we will assume that $v \notin \Delta(x,y,z)$.

\medskip\noindent {\sf Case 2:}
$v \in \St_x(y',z') \cup \St_y(x',z') \cup \St_z(x',y')$.

Suppose without loss of generality that $v$ belongs to the shade
$\St_x(y',z')$. If $x' \in [x,v]$, then
$d(x',v) = d(x,v) - d(x,x') \leq \delta$ and we are done since the
diameter is at most $2\delta$, hence we assume from now that
$x' \notin [x,v]$. We have $[x,v] \cap [y',z'] \ne \varnothing$. Then
by Lemma~\ref{lemma:triangle-shade} $[x,v]$ intersects one of the
sides $[z',x']$ and $[y',x']$ of $\Delta(x',y',z')$, say $[z',x']$.
But then $[x,v]$ intersects
$\partial B_{\delta}(x)\cap \Delta(x',y',z')$ in a point $v'$ and
$\partial B_{\delta}(y)\cap \Delta(x',y',z')$ in a point $v''$, where
$v'\in [x,v'']$. Since $d(v,x)\le 2\delta$ and $d(x,v')=\delta$, we
conclude that $d(v,v'')\le d(v,v')\le 2\delta-\delta=\delta$.

Next, we assert that $[y,x'] \cap [v,v''] \ne \varnothing$.  Let $s(x')$
be a shadow of $x'$ on $[y,z]$; we may assume that $[x,v] \cap
[y,s(x')] \neq \varnothing$. Then considering $\Delta(y,s(x'),x')$,
the geodesic $[x,v]$ intersects another of its side, either $[y,x']$ or
$[x',s(x')]$. In the latter case, it follows that $x' \in [x,v]$ and
we excluded that case. Hence we can assume the former case. Then as
$v''$ is not in the interior of $\Delta(y,x',z)$ by
Claim~\ref{claim:seven-triangles}. $[v'',v] \cap [y,x'] \neq
\varnothing$, as asserted.

Hence, we can suppose that $[y,x']\cap [v,v'']\ne \varnothing$. By Lemma
~\ref{quadrangle}, $d(y,v'')+d(v,x')\le d(y,x')+d(v,v'')$. Since
$d(y,v'')=d(y,x')=\delta$ and $d(v,v'')\le \delta$, we obtain that
$d(v,x')\le \delta$, concluding the proof of Case 2.

\medskip\noindent {\sf Case 3:}
$v\in \St_x(x',y',z') \cup \St_y(x',y',z') \cup \St_z(x',y',z')$.

Suppose without loss of generality that $v \in \St_x(x',y',z')$. In
view of Case 2, we can assume that $v \notin \St_x(y',z')$. By
Lemma~\ref{lemma:triangle-shade}
$v \in \St_x(x',y') \cap \St_x(x',z')$. By definition of
$\St_x(x',z')$, there is a line $(x,v)$ passing via $x$ and $v$ and
separating $x'$ from $z'$. Let $v' \in [x',z'] \cap (x,v)$. Let
$s(v')$ be a shadow of $v'$ in $[y,z]$ such that
$s(v') \in [x,v] \cap [y,z]$ (it exists because
$v' \in [x,v]$). Notice that $s(v') \notin \St_x(y',z')$. Indeed,
otherwise there exists a line $(x,s(v'))$ extending $[x,s(v')]$ and
separating the points $y'$ and $z'$.  But then $[x,s(v')]$ separates
$y'$ and $z'$ in $\Delta(x,y,z)$.  Therefore any line extending
$[x,s(v')]$, in particular the line $(x,v)$, also separates the points
$y'$ and $z'$. This contradicts the assumption
$v \notin \St_x(y',z')$. Hence $s(v') \notin \St_x(y',z')$.

Consider the shadows $s(x'),s(y'),$ and $s(z')$ of $x',y',$ and $z'$ in $[y,z]$
such that $s(v')\in [s(x'),s(z')]$ (such shadows
$s(x')$ and $s(z')$ exist by Claim~\ref{claim:shadow-convexity}).
Since $\St_x(y',z')$ is convex (Lemma~\ref{shade}) and $s(v') \notin
\St_x(y',z')$, we conclude that $s(v')$ does not belong to
$[s(z'),s(y')]$.  By Claim~\ref{claim:order-shadows}, either $s(v')$
belongs to $[y,s(z')]$ or $s(v')$ belongs to $[s(y'),z]$, say the
first. Consequently, further we will assume that $s(v') \in [y,s(z')]$
and $s(v') \ne s(z')$.  We have:
\begin{itemize}
\item[(i)] $d(y,v') \leq \delta$, because $v' \in [x',z']$, $d(y,x') =
d(y,z') = \delta$ and $B_{\delta}(y)$ is convex (Lemma~\ref{ball-convexity}),
\item[(ii)] $d(v,v') = d(v,x) - d(v',x) \leq 2\delta - \delta = \delta$,
because $v' \in [v,x]$ and $d(v',x) > \delta$ by minimality of
$\pi(x',y',z') \geq \pi(x',y',v')$.
\end{itemize}

Assume now that $x' \in \Delta(y,v,v')$. Then applying
Lemma~\ref{perimeter} to the triangles $\Delta(y,v,x')$ and
$\Delta(y,v,v')$ having $[y,v]$ as a side, we obtain $d(y,x') +
d(x',v) \leq d(y,v') + d(v',v)$. Since $d(y,x')=\delta$ and $d(y,v'),
d(v',v) \le \delta$, we derive that $d(x',v) \leq \delta$.

It remains to prove that $x' \in \Delta(y,v,v')$. We prove this in
two steps.  First, we show that $x' \in \Delta(x,y,v)$. Since
$s(v')\in [y,s(z')]$ and $s(v')\ne s(z')$, the point $v'$ belongs to
$\Delta(y,x,s(z'))\setminus [x,s(z')]$. Since $v'\in [x',z']$, we
conclude that $x'$ also belongs to $\Delta(y,x,s(z'))\setminus
[x,s(z')]$. Moreover, since $s(v')\in [s(x'),s(z')]$, the point
$s(x')$ is located between $y$ and $s(v')$. Since $x'\in [s(x'),x],$
$x'$ belongs to the triangle $\Delta(x,y,s(v'))$ and therefore to the
triangle $\Delta(x,y,v)$.

Second, we prove by way of contradiction that $x' \notin
\Delta(x,y,v')$. Otherwise, if $x'\in \Delta(x,y,v')$, let $z''$
be a point in the intersection of $[x,y]$ and a geodesic line
extending $[x',v']\subseteq [x',z']$.  Then $v' \in [z',z''] \subset
\Delta(x,y,z')$.  Applying Lemma~\ref{perimeter} to the triangles
$\Delta(x,y,z')$ and $\Delta(x,y,x')$ having $[x,y]$ as a side, we
get $2 \delta < d(y,x') + d(x',x) \leq d(y,z') + d(z',x) = 2\delta$, a
contradiction. This shows that indeed $x' \in \Delta(y,v,v')$ and
concludes the proof of Case 3.

\medskip\noindent{\sf Case 4:}
$v \notin \St_x(x',y',z') \cup \St_y(x',y',z') \cup \St_z(x',y',z')$.

Suppose without loss of generality that $v$ is separated from
$\Delta(x',y',z')$ by a line $(y,z)$ extending $[y,z]$ (such a line
exists by Lemma~\ref{separa}).  Suppose also by way of contradiction
that $v \notin B_{\delta}(x') \cup B_{\delta}(y') \cup B_{\delta}(z')$.
Since the shade $\St_x(x',y',z')$ is convex by Lemma
~\ref{shade}, the intersection of $\St_x(x',y',z')$ with
$(y,z)$ (and with $[y,z]$) is a geodesic segment $[p,q]$.  Let $v' \in
[x,v] \cap (y,z)$.  We assert that $v' \notin [p,q]$.  Indeed, if $v' \in
[p,q]$, then $v' \in \St_x(x',y',z')$, thus the intersection
$[x,v'] \cap \Delta(x',y',z')$ is nonempty. Since $[x,v'] \subseteq
[x,v]$, we conclude that $[x,v] \cap \Delta(x',y',z') \ne \varnothing$,
contrary to our assumption that $v \notin
\St_x(x',y',z')$. Consequently, $v' \notin [p,q]$.  Then one
can easily see that either $[p,q] \subseteq [y,v']$ or $[p,q] \subseteq
[v',z]$ holds, say the first. In this case, since
$s(x'),s(y'),s(z') \in [p,q],$ and $x' \in [x,s(x')]$, $y' \in [x,s(y')]$,
$z'\in [x,s(z')]$, we deduce that $x',y',z' \in \Delta(x,y,v)$. This
shows that either $\Delta(x',y',z') \subseteq \Delta(x,y,v)$ or
$\Delta(x',y',z') \subseteq \Delta(x,z,v)$ holds, say the first.

Let $\Delta(x'',y'',v'')$ be the critical triangle of the triplet
$x,y,v$. We assert that $x',y',z'\in \Delta(x'',y'',v'')$. For this
we will first prove that
\begin{displaymath}
\Delta(x,y,v) \setminus (B^{\circ}_{\delta}(x) \cup B^{\circ}_{\delta}(y) \cup B^{\circ}_{\delta}(v))
  \subseteq \Delta(x'',y'',v'').
\end{displaymath}
Indeed, since $d(y,x'') = d(y,v'') = \delta$ and the balls are convex,
$\Delta(y,v'',x'') \subseteq B_{\delta}(y)$. Moreover,
$\Delta(y,v'',x'') \setminus [v'',z''] \subseteq
B^{\circ}_{\delta}(y)$.  Indeed, any point $p \in
\Delta(y,v'',x'') \setminus [v'',z'']$ belongs to a geodesic segment
$[y,q]$ with $q \in [v'',z'']$. Since $q \in B^{\circ}_{\delta}(y)$ and
$p \ne q$, necessarily $d(y,p) < \delta$. Analogously, we obtain that
$\Delta(y'',x'',v) \setminus [y'',x''] \subseteq
B_{\delta}^{\circ}(v)$ and $\Delta(x,v'',y'') \setminus
[v'',y''] \subseteq B_{\delta}^{\circ}(x)$. On the other hand, each of
the triangles $\Delta(x,y,v''), \Delta(x,y'',v),$ and
$\Delta(x'',y,v)$ is covered by two of the three open balls
$B^{\circ}_{\delta}(x), B^{\circ}_{\delta}(y),$ and
$B^{\circ}_{\delta}(v)$. For example, $\Delta(x,y,v'')$ is covered
by $B^{\circ}_{\delta}(x)$ and $B^{\circ}_{\delta}(y)$. Indeed, by
monotonicity of perimeters (Lemma~\ref{perimeter}), for any point
$p \in \Delta(x,y,v'')$, we have $\min \{d(p,x), d(p,y)\} \le
\delta$. Moreover, by the same result, if $d(x,p) = \delta$, then
$d(y,p) < \delta$. This establishes that $\Delta(x,y,v'') \subseteq
B^{\circ}_{\delta}(x) \cup B^{\circ}_{\delta}(y)$. Now, the required
inclusion follows from Claim~\ref{claim:seven-triangles}.

Since $z'$ has distance $\delta$ to $x$ and $y$ and $z'$ has distance
$> \delta$ to $z$ and $v$, from previous inclusion we obtain
$z' \in \Delta(x'',y'',v'')$. Analogously, since $x'$ has distance
$\delta$ to $y$ and $z$ and $x'$ has distance $> \delta$ to $x$ and
$v$, we conclude that $x'\in \Delta(x'',y'',v'')$ (the proof for $y'$
is analogous). Hence $x',y',z'\in \Delta(x'',y'',z'')$. From
Lemma~\ref{perimeter} we conclude that
$\pi(x',y',z') < \pi(x'',y'',v'')$, contrary to the choice of the
triplet $x,y,z$ as a triplet having a critical triangle
$\Delta(x',y',z')$ of maximal perimeter. This concludes the proof of
Claim~\ref{claim:ball-cover} and of
Proposition~\ref{prop:three-balls}.
\end{proof}
\end{proof}

\subsection{Proof of Proposition~\ref{prop:twenty-three-balls}} \label{subsection:twenty-three-balls}

We start by restating Proposition \ref{prop:twenty-three-balls}:

\medskip\noindent
{\bf Proposition 3}. {\it
 Let $S$ be a compact subset of a Busemann surface $({\mathcal S},d)$
and let $u,v\in S$ be a diametral pair of $S$. Then $B_{2{\delta}}(v)\cap S$
can be covered by \constant{} balls
of radius ${\delta}$.
}

\begin{proof}
Let $S$ be a compact subset of a Busemann surface $({\mathcal S},d)$.
Let $u,v$ be a diametral pair of $S$, i.e., $u,v\in S$ and
$d(u,v)=\mbox{diam}(S)$. Let $\ell:=(u,v)$ be a line extending $[u,v]$ and
let $S'$ and $S''$ be the intersections of $S$ with the closed
halfplanes $\Pi'_{\ell}$ and $\Pi''_{\ell}$ defined by $\ell$. We will
show how to cover each of the sets $S'_0:=S'\cap B_{2{\delta}}(v)$ and
$S''_0:=S''\cap B_{2{\delta}}(v)$ with a fixed number of balls of
radius ${\delta}$.  We will establish this for $S'_0$, the same method
works for $S''_0$; at the end we will optimize over the two solutions
since some balls from different solutions have the same centers and
thus coincide.

If diam$(S)\le 2{\delta}$, we simply apply Proposition
~\ref{prop:three-balls}. Therefore, further we will assume that
diam$(S)>2{\delta}$. By Lemma~\ref{diameter}, $u,v$ is also a
diametral pair of conv$(S)$ and of conv$(S')$. Let $x$ be a point of
$[u,v]$ at distance $2{\delta}$ from $u$. Let $w$ be a point of
conv$(S')\cap \partial B_{2\delta}(v)$ maximizing the distance to $u$,
i.e., maximizing the perimeter $\pi(u,v,w)$.
Such a point $w$ exists because the set conv$(S')\cap \partial
B_{2\delta}(v)$ is compact and nonempty (the point $x$ belongs to this
intersection).

Let $x'$ be a point of $[u,w]$ at distance $\left(1 -
  \frac{2{\delta}}{d(u,v)}\right) d(u,w)$ from $w$. Notice that since
$d(u,w)\le d(u,v)$, we have $d(x',w)\le 2{\delta}$.  Notice also that
if we set $t:=1-\frac{2{\delta}}{d(u,v)}$, then $0<t<1$ and $x$ is the
point of $[u,v]$ such that $d(u,x)=t\cdot d(u,v)$ and $x'$ is the
point of $[u,w]$ such that $d(u,x')=t\cdot d(u,w)$. By Proposition A(iv)
$d(x,x')\le t\cdot d(v,w)<d(v,w)\le 2{\delta}$. On
the other hand, $d(u,x)-d(u,x')=t\cdot(d(u,v)-d(u,w))\ge 0$. Since
$d(x,v)=2{\delta}$, we conclude that $d(v,x')\ge 2{\delta}$ and
equality $d(v,x')=2{\delta}$ holds if and only if $x=x'$ (because in
case of equality, $x$ and $x'$ belong to the geodesic $[u,v]$ and thus
they must coincide).  Let $A$ be the quadrilateral of
$\Delta(u,v,w)$ bounded by the four geodesics $[x,x'],[x',w],[w,v],$
and $[v,x]$.

\begin{claimprop3} \label{claim:first}
$\Delta(u,v,w) \cap S'_0 = A \cap S'_0$.
\end{claimprop3}

\begin{proof} Indeed, suppose by way of contradiction that there exists a point
$z\in \Delta(u,v,w)\cap S'_0$ not belonging to $A$. Let $z'$ be a
point obtained as the intersection of $[x,x']$ with the extension of
the geodesic $[u,z]$ through $z$. Then $z\in [u,z']$ and $z'\ne z$,
yielding $d(u,z)<d(u,z')$. Since $d(u,z')\le \max\{
d(u,x),d(u,x')\}=d(u,x)$ by the convexity of balls, we deduce that
$d(u,z)<d(u,x)$. Since $d(v,z)\le 2{\delta}$ and $d(v,x)=2{\delta}$,
we conclude that $d(u,z)+d(z,v)<d(u,x)+d(x,v)$, contrary to the choice
of $x$ from $[u,v]$. This finishes the proof of Claim~\ref{claim:first}.
\end{proof}
\medskip

Let $B$ be the region of the halfplane $\Pi'$ consisting of
all points $z$ such that $[u,z]\cap [v,w]\ne \varnothing$. Finally,
let $C$ be the region of $\Pi'$ consisting of all points $z$ such that
$[z,v] \cap [u,w] \ne \varnothing$. Notice that $B \cup C$ consists of
precisely those points $z$ of $\Pi'$ such that $\Delta(u,v,w)$ and
$\Delta(u,v,z)$ are not comparable.

\begin{claimprop3} \label{claim:divide-regions}
$S'_0 \subseteq A \cup B \cup C$.
\end{claimprop3}

\begin{proof}   Using the remark preceding the statement, suppose by way of
  contradiction that $S'_0 \cap \Pi'$ contains a point $z$ such that
  $\Delta(u,v,w)$ is properly included in $\Delta(u,v,z)$. If
  $d(v,z) = 2 \delta$, then $\pi(u,v,w) < \pi(u,v,z)$ by Lemma
  ~\ref{perimeter}, and we will obtain a contradiction with the choice
  of $w$. Thus $d(v,z) < 2\delta$. Since $u \notin B_{2\delta}(v)$,
  the geodesic $[u,z]$ intersects $\partial B_{2\delta}(v)$ in a point
  $w'$. Let $w''$ be a common point of $[u,z]$ and a geodesic
  extension $(v,w)$. Then $w \in [w'',v]$. Since $d(v,w) = 2\delta$,
  we have $d(v,w'') > 2\delta$. Since $w'$ and $w''$ are located on
  $[u,z]$, $d(v,z) \le 2\delta$, and $d(v,w') = 2 \delta$, the
  convexity of $B_{2\delta}(v)$ implies that $w'$ is located on
  $[u,z]$ between $w''$ and $z$. This means that $\Delta(u,v,w)$ is
  properly contained in $\Delta(u,v,w')$. By Lemma~\ref{perimeter},
  $\pi(u,v,w') > \pi(u,v,w)$. Now, since
  $w' \in [z,u], d(v,w') = 2\delta,$ and $z \in S'_0$, we conclude
  that $w' \in \mbox{conv}(S') \cap \partial B_{2\delta}(v)$,
  contradicting the choice of $w$. This finishes the proof of
  Claim~\ref{claim:divide-regions}.
\end{proof}

Now, we will analyze how to cover the points of $S'_0$ in each of the
regions $A,B,C$.

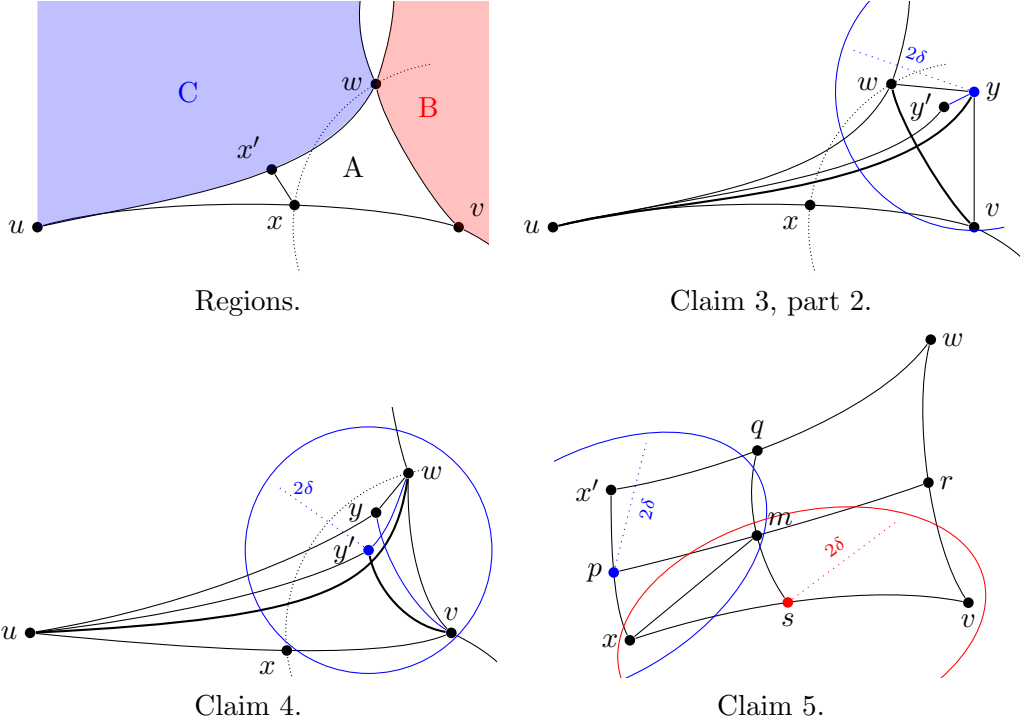
\begin{figure}\begin{tabular}{cc}
\begin{tikzpicture}[x=0.4cm,y=0.4cm]
\path[clip,use as bounding box] (-13.5,-12) rectangle  (2.5,-3);

\node [stpoint] (u) at (-12.5,-10.5) {};
\node [stpoint] (v) at (1.5,-10.5) {};
\draw[name path = sphere,densely dotted] (v) +(100:5.5) arc (100:240:5.5);
\path[name path = larger_sphere] (v) +(100:6.5) arc (100:240:6.5);
\draw (v) node [stpoint] (w) at +(120:5.5) {};

\draw[name path = geodesic_uv] (u) .. controls (-9,-9.5) and (-2,-9.5) .. (v);
\draw[name path = geodesic_uw] (u) .. controls (-8,-9.5) and (-3,-9) .. (w);

\draw[name intersections = {of = geodesic_uv and sphere}] (intersection-1) node (x) [stpoint] {};
\draw[name intersections = {of = geodesic_uw and larger_sphere}] (intersection-1) node (x') [stpoint] {};

\draw (x) -- (x');
\draw (v) .. controls (0.5,-9.5) and (-1,-7) .. (w);
\draw (w) .. controls (-0.5,-3.5) and (-0.5,-2) .. (-1,-1);
\draw (w) .. controls (-2,-4) and (-2,-2.5) .. (-1,-0.5);
\draw (v) .. controls (3.5,-11.5) and (4,-12.5) .. (4,-13.5);

\fill[blue, nearly transparent] (u.center) .. controls (-8,-9.5) and (-3,-9) .. (w.center)
      .. controls (-2,-4) and (-2,-2.5) .. (-1,-0.5) -- (-12.5,-1) -- cycle;

\fill[red,nearly transparent] (v.center)
  .. controls (3.5,-11.5) and (4,-12.5) .. (4,-13.5) -- (4,-0.5)
  -- (-1,-1) .. controls (-0.5,-2) and (-0.5,-3.5) .. (w.center)
  .. controls (-1,-7) and (0.5,-9.5) .. (v.center);

\path (u) node[left=1pt] {$u$};
\path (v) node[above right=0.5pt] {$v$};
\path (w) node[left=1pt] {$w$};
\path (x) node[below left=0.5pt] {$x$};
\path (x') node[above left = 0.5pt] {$x'$};

\node[text=blue] (textC) at (-7.5,-6) {C};
\node[text=red] (textB) at (0.5,-6.5) {B};
\node[text=black] (textA) at (-2,-8.5) {A};
\end{tikzpicture}

& \begin{tikzpicture}[x=0.4cm,y=0.4cm]
\path[clip,use as bounding box]  (3,-3) rectangle (-13.5,-12);

\node [stpoint] (u) at (-12.5,-10.5) {};
\node [stpoint] (v) at (1.5,-10.5) {};
\draw[name path = sphere,densely dotted] (v) +(100:5.5) arc (100:240:5.5);
\path[name path = larger_sphere] (v) +(100:6.5) arc (100:240:6.5);
\draw (v) node [stpoint] (w) at +(120:5.5) {};

\draw[name path = geodesic_uv] (u) .. controls (-9,-9.5) and (-2,-9.5) .. (v);
\draw[name path = geodesic_uw] (u) .. controls (-8,-9.5) and (-3,-9) .. (w);

\draw[name intersections = {of = geodesic_uv and sphere}] (intersection-1) node (x) [stpoint] {};
\draw[thick] (v) .. controls (0.5,-9.5) and (-1,-7) .. (w);

\draw (w) .. controls (-0.5,-3.5) and (-0.5,-2) .. (-1,-1);
\draw (v) .. controls (3.5,-11.5) and (4,-12.5) .. (4,-13.5);

\path (u) node[left=1pt] {$u$};
\path (v) node[above right=0.5pt] {$v$};
\path (w) node[left=1pt] {$w$};
\path (x) node[below left=0.5pt] {$x$};

\node [stpoint,fill=blue] (y) at (1.5,-6) {};
\draw[thick] (u) .. controls (-8.5,-9.5) and (-0.5,-9.5) .. (y);

\draw  (w) -- (y);\draw  (y) -- (v);
\node[stpoint] (y') at (0.5,-6.5) {};
\draw[thin] (u) .. controls (-8.5,-9.5) and (-1.5,-9) .. (y');

\path (y) node[right=0.5pt] {$y$};
\path (y') node[left=0.5pt] {$y'$};

\draw[blue,name path = sphere_y] (2.5,-10.5) arc (-77.4712:-240:4.6098);
\path[name path = help_line_radius] (-3.5,-4.5) -- (-2.5,-4.5);
\draw[name intersections = {of = sphere_y and help_line_radius}] (intersection-1) node (r) {};
\draw[dotted,blue] (y) -- (r) node[above,pos=0.5,sloped] {\tiny $2 \delta$};
\draw[blue] (y) -- (y');
\end{tikzpicture}
  \\
Regions. & Claim~\ref{claim:B-covered}, part 2. \\

\begin{tikzpicture}[x=0.4cm,y=0.4cm]
\path[clip,use as bounding box]  (3,-3) rectangle (-13.5,-12);

\node [stpoint] (u) at (-12.5,-10.5) {};
\node [stpoint] (v) at (1.5,-10.5) {};
\draw[name path = sphere,densely dotted] (v) +(100:5.5) arc (100:240:5.5);
\path[name path = larger_sphere] (v) +(100:6.5) arc (100:240:6.5);
\draw (v) node [stpoint] (w) at +(105:5.5) {};

\draw[name path = geodesic_uv] (u) .. controls (-7.5,-11) and (-1,-11.5) .. (v);
\draw[name path = geodesic_uw,thick] (u) .. controls (-2.5,-10) and (-0.5,-9) .. (w);
\draw[name intersections = {of = geodesic_uv and sphere}] (intersection-1) node (x) [stpoint] {};

\draw (v) .. controls (0.5,-9.5) and (0,-7.5) .. (w);
\draw (w) .. controls (-0.5,-3.5) and (-0.5,-2) .. (-1,-1);
\draw (v) .. controls (3.5,-11.5) and (4,-12.5) .. (4,-13.5);

\path (u) node[left=1pt] {$u$};
\path (v) node[above=1pt] {$v$};
\path (w) node[right=1pt] {$w$};
\path (x) node[below left=0.5pt] {$x$};

\node [stpoint] (y) at (-1,-6.5) {};
\draw (u) .. controls (-6.5,-9.5) and (-2.5,-7.5) .. (y);

\node[stpoint,fill=blue] (y') at (-1.25,-7.75) {};

\draw (u) .. controls (-8,-10) and (-3.5,-9) .. (y');
\path (y) node[left=1pt] {$y$};
\path (y') node[left=1pt] {$y'$};

\draw[name path = sphere_y,blue] (y') circle[radius=4.1];
\path[name path = help_line] (-5,-5) -- (-4.25,-5.75);
\draw[name intersections = {of = sphere_y and help_line}] (intersection-1) node (r) {};
\draw[dotted,blue] (y') -- (r) node[pos=0.7,above] {\tiny $2\delta$};

\draw  (y) edge (w);
\draw[blue] (y') .. controls (-0.5,-7) and (-0.25,-6) .. (w);
\draw[thick] (y') .. controls (-1,-9) and (0,-10.25) .. (v);

\draw[blue] (y) .. controls (-0.75,-8) and (0.25,-9.75) .. (v);
\end{tikzpicture}

&\begin{tikzpicture}[x=0.5cm,y=0.5cm]
\path[clip,use as bounding box]  (5,4) rectangle (-6.5,-5.5);

\node [stpoint] (x) at (-4.5,-4.5) {};
\node [stpoint] (v) at (4.5,-3.5) {};
\node [stpoint] (x') at (-5,-0.5) {};
\node [stpoint] (w) at (3.5,3.5) {};

\draw[name path = geodesic_xx] (x) .. controls (-5,-3.5) and (-5,-1.5) .. (x');
\draw[name path = geodesic_xv] (x) .. controls (-1.5,-3.5) and (2,-3) .. (v);
\draw[name path = geodesic_xw] (x') .. controls (-2,0) and (2,1.5) .. (w);
\draw[name path = geodesic_vw] (v) .. controls (3.5,-2) and (3,1.5) .. (w);

\path[name path = vertical_help] (-1.5,2.5) -- (0,-5);
\path[name path = horizontal_help] (-6,-3) -- (4.5,0);

\path[name intersections = {of = geodesic_xx and horizontal_help}]
  (intersection-1) node[stpoint,fill=blue] (p) {};
\path[name intersections = {of = geodesic_vw and horizontal_help}]
  (intersection-1) node[stpoint] (r) {};
\path[name intersections = {of = geodesic_xv and vertical_help}]
  (intersection-1) node[stpoint,fill=red] (s) {};
\path[name intersections = {of = geodesic_xw and vertical_help}]
  (intersection-1) node[stpoint] (q) {};

\draw[name path = geodesic_sq] (s) .. controls (-1,-2.5) and (-1.5,-1) .. (q);
\draw[name path = geodesic_pr] (p) .. controls (-2,-2) and (1.5,-1) .. (r);
\path[name intersections = {of = geodesic_sq and geodesic_pr}] (intersection-1) node[stpoint] (m) {};

\draw (x) -- (m);

\draw[name path = blue_sphere,rotate=30,blue]  (-6,0.5) ellipse (5 and 3);
\draw[name path = red_sphere,rotate=15,red]  (-1,-4) ellipse (5 and 3);
\path[name path = blue_help] (-4,1.5) -- (-4,0.5);
\path[name path = red_help] (3,-1) -- (2.5,-1.5);
\path[name intersections = {of = blue_sphere and blue_help}] (intersection-1) node (blue) {};
\path[name intersections = {of = red_sphere and red_help}] (intersection-1) node (red) {};

\draw[dotted,blue] (p) -- (blue) node[below,pos=0.5,sloped] {\tiny $2\delta$};
\draw[dotted,red] (s) -- (red) node[above,pos=0.5,sloped] {\tiny $2\delta$};

\path (x) node[left=0.5pt] {$x$};
\path (v) node[below=0.5pt] {$v$};
\path (w) node[right=0.5pt] {$w$};
\path (x') node[left=0.5pt] {$x'$};
\path (p) node[left=0.5pt] {$p$};
\path (q) node[above=0.5pt] {$q$};
\path (r) node[right=0.5pt] {$r$};
\path (s) node[below=0.5pt] {$s$};
\path (m) node[above right=0.3pt] {$m$};

\end{tikzpicture}
 \\
Claim~\ref{claim:C-covered}. & Claim~\ref{claim:A-covered}.\\
\end{tabular}
\caption{Illustration of the proof of Proposition~\ref{prop:twenty-three-balls}.}
\end{figure}

\begin{claimprop3} \label{claim:B-covered}
$\mathrm{diam}(B \cap S'_0)\le 2{\delta}$.
\end{claimprop3}

\begin{proof}   Pick any two points $y,y'\in B\cap S'_0 \le 2{\delta}$. If the
  triangles $\Delta(u,v,y)$ and $\Delta(u,v,y')$ are incomparable,
  i.e., $y\notin \Delta(u,v,y')$ and $y'\notin \Delta(u,v,y)$,
  then $[y,v]\cap [y',u]\ne \varnothing$ or $[y',v]\cap
  [y,u]\ne\varnothing$, say the first (this dichotomy follows from the
  fact that $\mathcal S$ is homeomorphic to ${\mathbb R}^2$). By
  Lemma~\ref{quadrangle}, $d(y,y')+d(u,v)\le d(y,v)+d(u,y')$. Since
  $d(y,v)\le 2{\delta}$ and $d(u,y')\le d(u,v)$ (by the choice of
  $v$), we conclude that $d(y,y')\le d(y,v)\le 2{\delta}$.

  Now, suppose that $y'\in \Delta(u,v,y)$. Since $[y,u]$ intersects
  $[v,w]$ and $d(u,y) \le d(u,v)$ by the choice of $v$, by
  Lemma~\ref{quadrangle} we have $d(y,w) \le d(v,w) = 2{\delta}$. Also
  $d(v,y) \le 2{\delta}$ because $y,y'\in S'_0$. Since $v, w \in
  B_{2{\delta}}(y)$, by the convexity of the ball $B_{2{\delta}}(y)$
  we conclude that $y' \in B_{2{\delta}}(y)$. Hence $d(y,y') \le
  2{\delta}$.  Consequently, $\mathrm{diam}(B\cap S'_0)\le 2{\delta}$.
\end{proof}

\begin{claimprop3} \label{claim:C-covered}
$\mathrm{diam}(C\cap S'_0)\le 2{\delta}$.
\end{claimprop3}

\begin{proof}   Pick any two points $y,y'\in C\cap S'_0$. Again, if the triangles
  $\Delta(u,v,y)$ and $\Delta(u,v,y')$ are incomparable, then we
  proceed as in the proof of Claim~\ref{claim:first}. Now suppose that
  $y'\in \Delta(u,v,y)$. Since $[y',v]$ intersects $[u,w]$ and
  $d(v,y') \le 2{\delta}$, $d(u,w)\le d(u,v)$, from
  Lemma~\ref{quadrangle} we deduce that $d(y',w) \le 2{\delta}$. Since
  $y'\in \Delta(u,v,y)\setminus \Delta(u,v,w)$, we conclude that
  $[y,v]\cap [y',w]\ne \varnothing$.  Again, by Lemma~\ref{quadrangle}
  $d(y,y')+d(v,w)\le d(y,v)+d(y',w)$. Since $d(v,w)=2\delta,$
  $d(y,v),d(y',w)\le 2{\delta}$, we immediately conclude that
  $d(y,y')\le 2\delta$.
\end{proof}

\begin{claimprop3} \label{claim:A-covered}
The set $A$ and consequently the set $S'_0 \cap A$ can be covered by
$4$ balls of radius ${\delta}$.
\end{claimprop3}

\begin{proof}   Recall that $A$ is a convex quadrilateral having all four sides
  $[x,x'],[x',w],[w,v],$ and $[v,x]$ of size at most $2{\delta}$. Let
  $p,q,r,$ and $s$ be the midpoints of $[x,x'],[x',w],[w,v],$ and
  $[v,x]$, respectively. By Proposition A(iii),   $d(p,r)\le
  2{\delta}$ and $d(q,s)\le 2{\delta}$. Let $m$ be the midpoint of
  $[q,s]$. Again, by Proposition A(iii),   $d(p,m)\le {\delta}$ and
  $d(m,r)\le {\delta}$. Since $d(m,q),d(m,s)\le {\delta}$, the
  geodesics $[m,p], [m,q],[m,r]$, and $[m,s]$ partition $A$ into four
  convex quadrilaterals with all sides at most ${\delta}$.

  We assert that $A$ is covered by the four ${\delta}$-balls centered
  at the points $p, q, r$ and $s$. Indeed, pick any point $z$ of
  $A$. Without loss of generality, we show that the quadrilateral with
  vertices $x,p,m,$ and $s$ is covered by $B_{\delta}(p)$ and
  $B_{\delta}(q)$. The geodesic $[x,m]$ splits this quadrilateral into
  two triangles $\Delta(x,p,m)$ and $\Delta(x,s,m)$. By convexity
  of balls, we have $\Delta(x,p,m) \subseteq B_{\delta}(p)$ and
  $\Delta(x,m,s) \subseteq B_{\delta}(s)$.
\end{proof}
\medskip

Summarizing, we conclude that $S'_0$ can be covered by $3+3+4 = 10$
balls of radius ${\delta}$. Analogously, the set $S''_0$ can be
covered by $10$ balls of radius ${\delta}$.  However, notice that the
ball $B_{\delta}(s)$ is counted in both coverings, thus $S\cap
B_{2{\delta}}(v)$ can be covered by $\constant$  balls of radius
${\delta}$. This finishes the proof of Proposition
~\ref{prop:twenty-three-balls}.
\end{proof}

\section{Open questions}

We conclude the paper with  three open questions.

\begin{question} \label{polygon}
Describe a polynomial time algorithm (in the number of sides and the
size of the packing) that, given a simple polygon $\mathcal P$ with
$n$ sides, constructs a covering and a packing of $\mathcal P$
satisfying the conditions of Corollary~\ref{cor:polygon}. Equivalently,
find a polynomial in $n$ algorithm (and maybe in the description of
$S$) to implement each step of the algorithm resulting from
Propositions 1-3: finding a covering of a closed subset $S$ of
$\mathcal P$ of diameter $\le 2\delta$ with at most 3 balls
(Proposition~\ref{prop:three-balls}) and the construction of the
regions $A,B,$ and $C$ in the proof of Proposition
~\ref{prop:twenty-three-balls}.
\end{question}

\begin{question} \label{polygon_holes}
Is it true that there exists a universal constant $c$ such that
$\rho(S)\le c\nu(S)$ for any compact (finite) subset of points of an
arbitrary polygon (with holes) endowed with the geodesic metric? Does
such a constant $c$ exist if diam$(S)\le 2\delta$, i.e., do polygons
with holes satisfy the weak-doubling property? The same questions can
be raised for polygons with holes on Busemann surfaces.
\end{question}

\begin{question} \label{busemann}
Is it true that the results of this note can be extended to all
2-dimensional Busemann spaces and, more generally, to all
$n$-dimensional Busemann spaces (in the latter case, the constant $c$
will depend of $n$)? The case of CAT(0) cube complexes (and, in particular, of
CAT(0) square complexes) is already interesting and nontrivial.
\end{question}

\bigskip\noindent{\bf Acknowledgments:} The authors would like to
thank the referees of this paper for careful reading of the previous 
versions and many useful remarks.

\bibliographystyle{plain}
\bibliography{covering}

\begin{thebibliography}{10}

\bibitem{AgMu}
Pankaj~K. {Agarwal} and Nabil~H. {Mustafa}.
\newblock {Independent set of intersection graphs of convex objects in 2D.}
\newblock {\em {Comput. Geom.}}, 34(2):83--95, 2006.

\bibitem{Al1}
N.~{Alon}.
\newblock {Piercing $d$-intervals.}
\newblock {\em {Discrete Comput. Geom.}}, 19(3):333--334, 1998.

\bibitem{Al}
Noga {Alon}.
\newblock {Covering a hypergraph of subgraphs.}
\newblock {\em {Discrete Math.}}, 257(2-3):249--254, 2002.

\bibitem{BaEdWo}
I.~{B\'ar\'any}, J.~{Edmonds}, and L.A. {Wolsey}.
\newblock {Packing and covering a tree by subtrees.}
\newblock {\em {Combinatorica}}, 6:221--233, 1986.

\bibitem{Be}
Claude {Berge}.
\newblock {\em {Hypergraphs. Combinatorics of finite sets. Transl. from the
  French.}}
\newblock Amsterdam etc.: North-Holland, 1989.

\bibitem{BoCh}
Glencora {Borradaile} and Erin~Wolf {Chambers}.
\newblock {Covering nearly surface-embedded graphs with a fixed number of
  balls.}
\newblock {\em {Discrete Comput. Geom.}}, 51(4):979--996, 2014.

\bibitem{Bou}
Nicolas Bousquet.
\newblock {\em {Hitting sets : VC-dimension and Multicut}}.
\newblock Theses, {Universit{\'e} Montpellier II - Sciences et Techniques du
  Languedoc}, December 2013.

\bibitem{BouTho}
Nicolas {Bousquet} and St\'ephan {Thomass\'e}.
\newblock {VC-dimension and Erd\H{o}s-P\'osa property.}
\newblock {\em {Discrete Math.}}, 338(12):2302--2317, 2015.

\bibitem{BrHa}
Martin~R. {Bridson} and Andr\'e {Haefliger}.
\newblock {\em {Metric spaces of non-positive curvature.}}
\newblock Berlin: Springer, 1999.

\bibitem{BrGo}
H.~{Br\"onnimann} and M.T. {Goodrich}.
\newblock {Almost optimal set covers in finite VC-dimension.}
\newblock {\em {Discrete Comput. Geom.}}, 14(4):463--479, 1995.

\bibitem{Ca}
Peter~J. {Cameron}.
\newblock {Problems from CGCS Luminy, May 2007.}
\newblock {\em {Eur. J. Comb.}}, 31(2):644--648, 2010.

\bibitem{ChChNa}
J\'er\'emie {Chalopin}, Victor {Chepoi}, and Guyslain {Naves}.
\newblock {Isometric embedding of Busemann surfaces into $L_1$.}
\newblock {\em {Discrete Comput. Geom.}}, 53(1):16--37, 2014.

\bibitem{ChHP}
Timothy~M. {Chan} and Sariel {Har-Peled}.
\newblock {Approximation algorithms for maximum independent set of
  pseudo-disks.}
\newblock {\em {Discrete Comput. Geom.}}, 48(2):373--392, 2012.

\bibitem{ChEs}
Victor {Chepoi} and Bertrand {Estellon}.
\newblock {Packing and covering $\delta $-hyperbolic spaces by balls.}
\newblock In {\em {Approximation, randomization, and combinatorial
  optimization. Algorithms and techniques. 10th international workshop, APPROX
  2007, and 11th international workshop, RANDOM 2007, Princeton, NJ, USA,
  August 20--22, 2007. Proceedings.}}, pages 59--73. Berlin: Springer, 2007.

\bibitem{ChEsVa}
Victor {Chepoi}, Bertrand {Estellon}, and Yann {Vaxes}.
\newblock {Covering planar graphs with a fixed number of balls.}
\newblock {\em {Discrete Comput. Geom.}}, 37(2):237--244, 2007.

\bibitem{ChFe}
Victor {Chepoi} and Stefan {Felsner}.
\newblock {Approximating hitting sets of axis-parallel rectangles intersecting
  a monotone curve.}
\newblock {\em {Comput. Geom.}}, 46(9):1036--1041, 2013.

\bibitem{Cl}
K.L. {Clarkson}.
\newblock {Nearest neighbor queries in metric spaces.}
\newblock {\em {Discrete Comput. Geom.}}, 22(1):63--93, 1999.

\bibitem{Co}
G\'erard {Cornu\'ejols}.
\newblock {\em {Combinatorial optimization. Packing and covering.}}
\newblock Philadelphia, PA: SIAM, Society for Industrial and Applied
  Mathematics, 2001.

\bibitem{CoFePe-LaSo}
Jos\'e {Correa}, Laurent {Feuilloley}, Pablo {P\'erez-Lantero}, and Jos\'e~A.
  {Soto}.
\newblock {Independent and hitting sets of rectangles intersecting a diagonal
  line: algorithms and complexity.}
\newblock {\em {Discrete Comput. Geom.}}, 53(2):344--365, 2015.

\bibitem{Du}
R.M. {Dudley}.
\newblock {\em {Uniform central limit theorems.}}
\newblock Cambridge: Cambridge University Press, 1999.

\bibitem{GuLe}
A.~{Gy\'arf\'as} and J.~{Lehel}.
\newblock {Covering and coloring problems for relatives of intervals.}
\newblock {\em {Discrete Math.}}, 55:167--180, 1985.

\bibitem{HaDe}
Hugo {Hadwiger} and H.~{Debrunner}.
\newblock {Kombinatorische Geometrie in der Ebene.}
\newblock {(Monographies de L'Enseignement math\'ematique. No. 2.) Gen\`eve:
  Institut de Math\'ematique, Universit\'e, 122 S. (1960).}, 1960.

\bibitem{Iv}
Sergei {Ivanov}.
\newblock {On Helly's theorem in geodesic spaces.}
\newblock {\em {Electron. Res. Announc. Math. Sci.}}, 21:109--112, 2014.

\bibitem{Bo}
K\'aroly jun. {B\"or\"oczky}.
\newblock {\em {Finite packing and covering.}}
\newblock Cambridge: Cambridge University Press, 2004.

\bibitem{Ka}
R.N. {Karasev}.
\newblock {Transversals for families of translates of a two-dimensional convex
  compact set.}
\newblock {\em {Discrete Comput. Geom.}}, 24(2-3):345--353, 2000.

\bibitem{KoTi}
A.N. {Kolmogorov} and V.M. {Tikhomirov}.
\newblock {$\varepsilon$-entropy and $\varepsilon$-capacity of sets in function
  spaces.}
\newblock {\em {Transl., Ser. 2, Am. Math. Soc.}}, 17:227--364, 1959.

\bibitem{KuMi}
Sanjeev~R Kulkarni.
\newblock On metric entropy, vapnik-chervonenkis dimension, and learnability
  for a class of distributions.
\newblock Technical report, DTIC Document, 1989.

\bibitem{Lo}
G.G. {Lorentz}.
\newblock {Metric entropy and approximation.}
\newblock {\em {Bull. Am. Math. Soc.}}, 72:903--937, 1966.

\bibitem{MuRa}
Nabil~H. {Mustafa} and Saurabh {Ray}.
\newblock {Improved results on geometric hitting set problems.}
\newblock {\em {Discrete Comput. Geom.}}, 44(4):883--895, 2010.

\bibitem{PaAg}
J\'anos {Pach} and Pankaj~K. {Agarwal}.
\newblock {\em {Combinatorial geometry.}}
\newblock New York, NY: John Wiley \& Sons, 1995.

\bibitem{Pa}
Athanase {Papadopoulos}.
\newblock {\em {Metric spaces, convexity and nonpositive curvature.}}
\newblock Z\"urich: European Mathematical Society Publishing House, 2005.

\bibitem{PoShRo}
R.~{Pollack}, M.~{Sharir}, and G.~{Rote}.
\newblock {Computing the geodesic center of a simple polygon.}
\newblock {\em {Discrete Comput. Geom.}}, 4(6):611--626, 1989.

\bibitem{Sch}
Alexander {Schrijver}.
\newblock {\em {Combinatorial optimization. Polyhedra and efficiency (volume
  B).}}
\newblock Berlin: Springer, 2003.

\bibitem{Va}
Vijay~V. {Vazirani}.
\newblock {\em {Approximation algorithms.}}
\newblock Berlin: Springer, 1999.

\bibitem{Vi}
Ivo Vigan.
\newblock Packing and covering a polygon with geodesic disks.
\newblock {\em arXiv preprint arXiv:1311.6033}, 2013.

\bibitem{We}
G.~{Wegner}.
\newblock {\"Uber eine kombinatorisch-geometrische Frage von Hadwiger und
  Debrunner.}
\newblock {\em {Isr. J. Math.}}, 3:187--198, 1965.

\end{thebibliography}

\end{document}